\documentclass[a4paper,10pt]{article}
\usepackage{amssymb}
\usepackage{amsthm}
\usepackage{amsmath}
\usepackage{url}
\usepackage{longtable}

\newtheorem{theorem}{Theorem}[section]
\newtheorem{lemma}[theorem]{Lemma}

\newtheorem{definition}[theorem]{Definition}
\newtheorem{remark}[theorem]{Remark}

\author{Philipp Grohs\footnote{TU Wien, Institute of Discrete Mathematics and Geometry, Wiedner Hauptstrasse 8-10, 1040 Wien, Austria. Email: philippgrohs@gmail.com, web: http://www.dmg.tuwien.ac.at/grohs, phone: +43 58801 11318.}}
\title{Continuous Shearlet Frames and Resolution of the Wavefront Set}
\date{}
\begin{document}

\maketitle
\begin{abstract}
In recent years directional multiscale transformations like the curvelet- or shearlet transformation
have gained considerable attention. The reason for this is that these transforms are - unlike more traditional transforms like wavelets - able to efficiently handle data with 
features along edges. 
The main result in \cite{Labate2009} confirming this property for shearlets is due to Kutyniok and Labate where it
is shown that for very special functions $\psi$ with frequency support in a compact conical wegde the decay rate 
of the shearlet coefficients of a tempered distribution $f$ 
with respect to the shearlet $\psi$ can resolve the Wavefront Set of $f$.
We demonstrate that the same result can be verified under much weaker assumptions on $\psi$, namely to
possess sufficiently many anisotropic vanishing moments. We also show how to build frames for
$L^2(\mathbb{R}^2)$ from any such function.  To prove our statements we develop a new approach based on  an adaption of the Radon transform
to the shearlet structure.
\end{abstract}
\tableofcontents
\newcommand{\Z}{{\mathbb Z}}
\newcommand{\N}{{\mathbb N}}
\newcommand{\R}{{\mathbb R}}
\newcommand{\schoen}{\mathcal}
\newcommand{\V}{{\schoen V}}
\newcommand{\F}{{\schoen F}}
\newcommand{\T}{{\schoen T}}
\newcommand{\A}{{\schoen A}}
\newcommand{\p}{{p}}
\newcommand{\q}{{\bf q}}
\newcommand{\h}{{\bf h}}
\newcommand{\e}{{\bf e}}
\newcommand{\Pp}{{\schoen P}}
\newcommand{\Ss}{{S}}
\newcommand{\Uu}{{\schoen U}}
\newcommand{\ii}{{\bf i}}
\renewcommand{\j}{{\bf j}}
\renewcommand{\k}{{\bf k}}
\renewcommand{\d}{{\bf d}}
\renewcommand{\l}{{\bf l}}
\newcommand{\C}{\mathcal{C}}
\newcommand{\D}{\mathcal{D}}
\section{Introduction}
\subsection{Previous work and notation}
One of the main themes of computational harmonic analysis is to represent a given function by its inner products with respect to a given set of 'atoms'. 
Prominent examples for such representations are the \emph{Gabor transform} \cite{Grochenig2000} or \emph{Wavelet Transforms} \cite{Daubechies1992}. 
A common feature of most of these 'classical' representations is that they are \emph{isotropic}, meaning that they treat every direction in space equally.
For this reason they are not able to efficiently and accurately handle directional phenomena at different scales. This is certainly a large drawback 
considering the fact that in many applications, like for example image processing, the main information lies precisely in the directional features (edges).

Until recently, one mostly had to resort to adaptive methods in order to capture such phenomena, but in 2003 Candes and Donoho managed a breakthrough in 
this problem by introducing the \emph{curvelet} transform \cite{Candes1999,candes2003,candes2003a} which can be used to analyze functions defined on the plane $\R^2$. 

The idea behind curvelets is to carve up the frequency (i.e. wavenumber- ) domain into circular wedges each corresponding to a scale 
$a$ and a direction $\theta\in (0,2\pi]$ obeying 
the \emph{parabolic scaling relation} 

\begin{equation}\label{eq:scalrel}
	width \sim length^2.
\end{equation}

Then, similarly to the Littlewood-Paley construction \cite{Frazier1991}, Candes and Donoho construct functions
$\gamma_{a,\theta,0}(x):=\gamma_{a,0,0}(\mbox{R}_\theta x),\ x\in \R^2$,
$\mbox{R}_\theta$ denoting the rotation matrix by $\theta$, which have frequency support in the wedge corresponding to scale $a$ and direction $\theta$. 
To also localize in space they use the translates of these functions and arrive at the family of 'atoms'

$$
	\gamma_{a,\theta,t}(x) = \gamma_{a,0,0}\big(\mbox{R}_\theta(x-t)\big),\quad a\in \R_+, \theta \in (0,2\pi],\ t\in \R^2,
$$

which correspond to scale $a$, location $t$ and direction $\theta$. 
In \cite{candes2003} they prove a representation formula which states that any $L^2$-function $f$ can be fully and
stably recovered from its \emph{curvelet coefficients} $\langle f , \gamma_{a,\theta,t}\rangle$.
Moreover, they show that the decay rates of the curvelet coefficients give precise information on the 
directional behavior of $f$ \cite{candes2003}. 

The idea of partitioning the plane into 'parabolic wegdes' is not new, by the way and already occurs in \cite{Fefferman1973, Seeger1991}, although in a very different context.
It is also featured in \cite{Stein1993} under the name 'Second Dyadic Decomposition'. In \cite{Borup2007} various discrete curvelet-like frame constructions are given
based on the same partition and the notion of \emph{decomposition spaces} \cite{Feichtinger1985,Feichtinger1987}.

In the construction of the curvelet system one starts with the construction of the functions $\gamma_{a,0,0}$ which correspond to scale $a$. 
In particular, curvelets do not form an \emph{affine system}, they are not generated from one (or finitely many) functions.

Looking at wavelet transforms it might seem natural to define 
$\gamma_{a,0,0} = D_{A_a}\psi$, where $D_{A_a}$ denotes the dilation operator with the expanding matrix $A_a = \mbox{diag}(a,a^{1/2})$ 
(this choice of dilation matrix reflects the parabolic scaling relation (\ref{eq:scalrel})),
and consider the system $\gamma_{a,\theta,t} = T_t D_{R_\theta} D_{A_a} \psi$, where we define for $t\in \R^2$ the \emph{translation operator} 
$T_t: f\mapsto f(\cdot - t)$ and for $M\in \mbox{GL}(2,\R)$
the \emph{dilation operator} $D_M: f\mapsto \det (M)^{-1/2}f(M^{-1}\cdot)$. This is what has essentially been done in \cite{Smith1998} and
goes by the name 'Hart Smith's transform' in \cite{candes2003}. 
Note that the mapping $(a,\theta,t)\mapsto T_tD_{R_\theta}D_{A_a}$ does not carry the useful structure of a group representation. This lack of structure 
makes it more difficult to constuct \emph{tight} frames for this system.

Remarkably, a group structure \emph{can} be achieved if one replaces the rotation transforms 

$$
	R_\theta  = \left(\begin{array}{cc}\sin(\theta) & \cos(\theta)\\\cos(\theta) & -\sin(\theta)\end{array}\right)
$$

by \emph{shear transforms} defined by

$$
	S_s:= \left(\begin{array}{cc}1 & -s\\0 &1\end{array}\right)
$$

and consider the system

\begin{equation}\label{eq:sys}
	\psi_{ast}:= T_t D_{S_s} D_{A_a}\psi
\end{equation}

for some \emph{shearlet} $\psi$ \cite{Dahlke2008,Dahlke2009}. The directional component is now encoded in the shear parameter $s\in \R$.

\begin{definition}
The \emph{Shearlet Transform} of a function (or distribution) $f$ defined on $\R^2$ with respect to a function (or distribution) $\psi$ defined on $\R^2$ is defined as
\begin{equation}\label{eq:shdef}
	\mathcal{SH}_\psi f(a,s,t):= \langle f , \psi_{ast}\rangle ,\quad a\in \R_+, s\in \R, t\in \R^2
\end{equation}
and 
\begin{equation}
	\psi_{ast}(x_1,x_2)=a^{-3/4}\psi \big(\frac{x_1-t_1 + s(x_2-t_2)}{a},\frac{x_2-t_2}{a^{1/2}}\big).
\end{equation}
\end{definition}

As stated before, the system (\ref{eq:sys}) can be regarded as the orbit of a function $\psi$ under
the action of a unitary representation of a group, the so-called shearlet group 
$\mathbb{S} = \big( \R_+\times \R \times \R^2,\circ\big)$ with
$$
	(a,s,t)\circ (\tilde a ,\tilde s , \tilde t ) = (a\tilde  a , s + \tilde s \sqrt a , t + S_s A_s \tilde t).
$$
Using this structure one can apply the machinery of square integrable group representations to obtain
representation formulas for $L^2$-functions.
We denote by $\hat f$ the \emph{Fourier transform} of a function $f\in L^2(\R^2)$ defined by
$$
	\hat f(\omega) = \int_{\R^2}f(x)e^{-2\pi i x\omega}dx.
$$

\begin{definition}\label{thm:adm}
A function $\psi$ is admissible if and only if 
\begin{equation}\label{eq:admdef}
	C_\psi:=\int_{\R^2}\frac{|\hat \psi (\omega)|^2}{|\omega_1|^2} d\omega = 
	\int_\R\int_{\R_+}|(D_{S_s}D_{A_a}\psi)^\land (\xi)|^2a^{-3/2}dads< \infty. 
\end{equation}
Admissible functions are called \emph{shearlets}.
\end{definition}

The equivalence of the second and third term in (\ref{eq:admdef}) can be seen by using the
substitution $\omega(a,s) = D_{S_s^{-T}}D_{A_a^{-T}}\xi$. The reason why this works (meaning that the integral is independent of $\xi$) is just the group 
structure of the shearlet group and the fact that the left Haar measure is given by $a^{-3}dadsdt$ \cite{Dahlke2008}.
The following notion will often be used in the sequel:

\begin{definition}\label{def:moments}
We shall say that $f$ has $n$-vanishing moments in $x_1$-direction if 
\begin{equation}\label{def:amoments}
	\int_{\R^2}\frac{|\hat \psi (\omega)|^2}{|\omega_1|^{2n}} d\omega < \infty.
\end{equation}
\end{definition}

\begin{remark}
The reason for the teminology 'vanishing moments' for the condition (\ref{def:amoments}) is as follows: if we assume sufficient spatial decay of
$\psi$, the condition (\ref{def:amoments}) is (almost) equivalent to 
$$
	\int_{\R}x_1^k\psi(x_1,x_2) dx_1 = 0\quad \mbox{for all }x_2\in \R,\ k < n
	$$
see e.g. \cite{Mallat2009} for similar statements related to wavelets. 
\end{remark}

With some Fourier analysis one can show the following:

\begin{theorem}
If $\psi$ be an admissible shearlet, then for all $f\in L^2(\R^2)$ we have the representation
formula
\begin{equation}\label{eq:reprgroup}
	\|f\|_2^2 = \int_{(a,s,t)\in \mathcal{SH}}|\mathcal{SH}_\psi (f)(a,s,t)|^2 a^{-3}da ds dt.
\end{equation}
\end{theorem}

From the admissibility condition we see that shearlets exist in abundance. In fact all they need to satisfy is
to have one vanishing moment in the $x_1$-direction or equivalently
to be a partial derivative in $x_1$-direction of a square integrable function.
We denote by $H_{(n_1,n_2)}(\R^2)$ the Sobolev space defined by
$$
	H_{(n_1,n_2)}(\R^2):=\{f\in L^2(\R^2):\ (\frac{\partial}{\partial x_1}\big)^{n_1}(\frac{\partial}{\partial x_2}\big)^{n_2}f \in L^2(\R^2)\}.
$$
\begin{theorem}\label{thm:shabl}
Let $\theta$ be in $H_{(n,0)}(\R^2)$ with $\hat \theta(0)\neq 0$. Then the function 
\begin{equation}\label{eq:der}
	\psi(x)=(-1)^n \big(\frac{\partial}{\partial x_1}\big)^n\theta(x)
\end{equation}
is a continuous shearlet with $n$ vanishing directional moments in $x_1$-direction.
Conversely, let $\psi$ be a continuous shearlet with $n$ vanishing moments. Then $\psi$
can be written in the form (\ref{eq:der}) with a function $\theta\in H_{(n,0)}(\R^2)$.
\end{theorem}

\begin{proof}This is an easy exercise.
\end{proof}

The representation (\ref{eq:reprgroup}) has one drawback: while in the curvelet transform the directional
parameter $\theta$ ranges over a compact set, we now need to consider all shear directions $s\in \R$. It is
also easy to see that as $s$ gets large the directional resolution gets denser and denser, so that 
the parameter $s$ does not distribute the directions uniformly. 
In order to avoid this dependence on the choice of the $x_1$ and $x_2$-coordinates, Labate et. al. proposed the following
construction called 'shearlets on the cone' in \cite{Labate2005}:

Let $f\in L^2(\R^2)$. Then decompose $f = P_{\D} f + P_{\C}f + P_{\C^\nu}f$ where $P_\D$ is a frequency projection onto $[-2,2]^2$, $P_\D$ onto 
$\C =\{\xi : \ |\xi_1|\geq 2 ,\ |\frac{\xi_2}{\xi_1}|\le 1\}$
and $P_{\C^\nu}$ onto $\C^{\nu} =\{\xi : \ |\xi_2|\geq 2 ,\ |\frac{\xi_1}{\xi_2}|\le 1\}$. It is well-known from microlocal analysis that directional 
singularities with slope $\geq 1$ manifest themselves
as slow decay in $\C$ and singularities with slope $\le 1$ as slow decay in $\C^{\nu}$. Therefore $P_\C  f$ can
be seen as the part of $f$ containing singularities with slope $\geq 1$, $P_{\C^\nu}f$ as the part of $f$ containing singularities with slope 
$\le 1$ and $P_\D f$ as a smooth low-pass approximation of $f$.

The 'shearlets on the cone'-construction essentially\footnote{we have modified the construction a bit in order to remain closer to our later results, 
for the precise construction we refer to \cite{Labate2009}} goes as follows, denoting by $C_0^\infty$ the space of rapidly decaying $C^\infty$-functions:

Let $\psi$ be defined via $\hat \psi(\xi) = \hat \psi_1(\xi_1)\hat\psi_2(\frac{\xi_1}{\xi_1})$ where
$\hat\psi_1\in C^\infty_0(\R)$ is a wavelet which has frequency support $\mbox{supp }\hat \psi_1 \subseteq [-2,-1/2]\cup [1/2,2]$ and $\|\psi_2\|_2 = 1$ 
and $\hat \psi_2\in C_0^\infty(\R)$ with 
$\mbox{supp }\hat \psi_2\subseteq [-1,1]$, and $\hat \psi_2 >0$ on $(-1,1)$. 

Then with $W$ a suitable window function one can show the following \cite[Equation (3.7)]{Labate2009}:

\begin{eqnarray}\nonumber
	\|f\|_2^2 &=& \int_{t\in \R^2}|\langle P_\D f , T_t W\rangle |^2 dt + \int_{t\in \R^2}\int_{-2}^2\int_0^1
	|\mathcal{SH}_\psi (P_\C f)(a,s,t)|^2 a^{-3}da ds dt \\
	& & + \int_{t\in \R^2}\int_{-2}^2\int_0^1
	|\mathcal{SH}_{\psi^\nu} (P_{\C^\nu} f)(a,s,t)|^2 a^{-3}da ds dt,\quad \mbox{for all }f\in L^2(\R^2),
	\label{eq:reprorig}
\end{eqnarray}
where we let $\hat \psi^\nu(\xi_1,\xi_2) := \hat \psi(\xi_2,\xi_1)$ and
$\mathcal{SH}_{\psi^\nu} f (a,s,t)= \langle f, \psi_{ast}^\nu\rangle$.

More importantly, Kutyniok and Labate have shown in \cite{Labate2009} that the shearlet coefficients
in the representation (\ref{eq:reprorig}) characterize the \emph{Wavefront Set} $\mbox{WF}(f)$ of a tempered distribution $f$, which roughly means the
set of points $t\in \R^2$ and directions $s\in \R$ along which $f$ is not smooth
at $t$. See the next section on more information regarding the Wavefront Set.
The result \cite[Theorem 5.1]{Labate2009} is as follows:

\begin{theorem}\label{thm:kut}
Let $\psi$ be constructed according to the 'shearlets on the cone'-construction.
Let $f$ be a tempered distribution and $\mathcal{D}=\mathcal{D}_1\cup\mathcal{D}_2$, where
$\mathcal{D}_1 = \{(t_0,s_0)\in \R^2 \times [-1,1]:$ for $(s,t)$ in a neighborhood $U$ of $(s_0,t_0)$, 
$|\mathcal{SH}_\psi f(a,s,t)| = O(a^k)$ for all $k\in \mathbb{N}$, with the implied constant uniform over $U\}$ and 
$\mathcal{D}_2 = \{(t_0,s_0)\in \R^2 \times (1,\infty]:$ for $(1/s,t)$ in a neighbourhood $U$ of $(s_0,t_0)$, 
$|\mathcal{SH}_{\psi^\nu} f(a,s,t)| = O(a^k)$ for all $k\in \mathbb{N}$, with the implied constant uniform over $U\}$.
Then 
$$
	\mbox{WF}(f)^c = \mathcal{D}.
$$
\end{theorem}

We would like to mention that curvelets and shearlets are not the only continuous transforms capable of
'resolving the Wavefront Set', see for example \cite{Cordoba1978,Sjostrand1982,Smith1998}.  From a computational point of view curvelets and shearlets have probably received the greatest attention recently (although many ideas are already contained in \cite{Smith1998}). 

{\bf A Word on Notation.} We shall use the symbol $|\cdot |$ indiscriminately for the absolute value on $\R, \R^2$ and $\mathbb{C}$. We usually denote vectors
in $\R^2$ by $x,t,\xi,\omega$ and their elements by $x_1,x_2,t_1,t_2,\dots$. In general it should always be clear to which space a variable belongs. The symbol $\|\cdot\|$ is reserved for various function space and operator norms.

%----------------
\subsection{Contributions}
%----------------

The motivation of the present work is the fact that on the one hand, by Theorem \ref{thm:shabl}
we know that shearlets exist in abundance, but on the other hand the 'shearlet on the cone'-construction described above is very specific. 
In particular it requires $\psi$ to have compact frequency support and thus infinite spacial support which might be undesirable for some applications
(the same caveat is valid for curvelets). 
So the question we would like to answer is: 'what is really needed for a function $\psi$ so that
a representation similar to (\ref{eq:reprorig}) and a result like Theorem \ref{thm:kut} are valid?'.
It turns out that there are no restrictions on $\psi$ besides the obvious ones, i.e. vanishing moments in 
the $x_1$-direction:

In Section \ref{sec:res} we show (among other things) the following theorem:

\begin{theorem}
Let $\psi$ be a Schwartz function with infinitely many vanishing moments in $x_1$-direction.
Let $f$ be a tempered distribution and $\mathcal{D}=\mathcal{D}_1\cup\mathcal{D}_2$, where
$\mathcal{D}_1 = \{(t_0,s_0)\in \R^2 \times [-1,1]:$ for $(s,t)$ in a neighbourhood $U$ of $(s_0,t_0)$, 
$|\mathcal{SH}_\psi f(a,s,t)| = O(a^k)$ for all $k\in \mathbb{N}$, with the implied constant uniform over $U\}$ and 
$\mathcal{D}_2 = \{(t_0,s_0)\in \R^2 \times (1,\infty]:$ for $(1/s,t)$ in a neighbourhood $U$ of $(s_0,t_0)$, 
$|\mathcal{SH}_{\psi^\nu} f(a,s,t)| = O(a^k)$ for all $k\in \mathbb{N}$, with the implied constant uniform over $U\}$.
Then 
$$
	\mbox{WF}(f)^c = \mathcal{D}.
$$
\end{theorem}

We also show an analogous result if $\psi$ has only finitely many vanishing moments in $x_1$-direction.

In addition, in Section \ref{sec:frame} we show that for essentially any shearlet a reproduction formula
similar to (\ref{eq:reprorig}) holds. To show this we shall use the notion of continuous frames.

These results provide a substantial generalization to the previous results in \cite{Labate2009}. In particular they allow for the first time 
to also use compactly supported functions $\psi$, for example tensor-product wavelets, in the analysis. 

We would like to remark that, while the methods of proof of the results \cite{Labate2009,candes2003} are quite similar to each other, our proofs follow different arguments.  
As a main tool we use a version of the Radon transform which is specifically adapted to the shear operation.

{\bf Outline.} The outline is as follows. In the next section, Section \ref{sec:wf}, we introduce the notion of $N$-Wavefront set. We also introduce the Radon 
transform which serves as an extremely convenient tool for our analysis. Then, in Section \ref{sec:wf} we show a
\emph{direct theorem} stating that $\mathcal{SH}_\psi f(a,s,t)$ has fast decay in $a$ if $f$ is smooth in 
$t$ and direction $s$. As already noted, in Section \ref{sec:frame} we derive several representation theorems for
$L^2$-functions based on the notion of continuous frames. In Section \ref{sec:inverse} we show an \emph{inverse theorem} stating that $f$ is smooth in $t$ and 
direction $s$ if $\mathcal{SH}_\psi f (a,s,t)$ has
fast decay in $a$. Section \ref{sec:res} summarizes these results and contains some theorems concerning the
resolution of the Wavefront Set.

%----------------------------------------------------------------------------
\section{The Wavefront set}\label{sec:wf}
%----------------------------------------------------------------------------

In this section we introduce and explain the notion of the Wavefront Set of a tempered distribution $f$ which has its roots in the analysis of the propagation of 
singularities of partial differential equations \cite{Hormander1983}, see also \cite{Toft2007} for a discussion of
various notions of Wavefront Sets. It is in some sense easier to observe directional 'microlocal' phenomena in the
Fourier domain and therefore the definition we give for the Wavefront Set is formulated in terms of the Fourier transform of a localized version of $f$. 
We shall however see in the part on the Radon transform that the so-called Projection Slice Theorem gives us a tool to study microlocal behavior of $f$ 
in the spacial domain -- by studying the (univariate) regularity of the Radon transform of a localized version of $f$.
%---------
\subsection{Definition}
%---------
We now give the definition of the Wavefront Set of a tempered distribution.

\begin{definition}
Let $N\in \R$ and $f$ tempered distribution on $\mathbb{R}^2$. We say that $x\in \mathbb{R}^2$ is an \emph{$N$-regular point} if there exists a neighbourhood 
$U_x$ of $x$ such that $\Phi \psi \in C^N$, where $\Phi$ is a smooth cutoff function with $\Phi\equiv 1$ on $U_x$. Furthermore, we call $(x,\lambda)$ an 
\emph{$N$-regular directed point} if there exists a neighbourhood $U_x$ of $x$, a smooth cutoff function $\Phi$ with $\Phi\equiv 1$ on $U_x$ and a neighbourhood 
$V_\lambda$ of $\lambda$ such that
\begin{equation}\label{eq:nwv}
	(\Phi f)^\land(\eta) = O\big((1 + |\eta|)^{-N}\big)\quad \mbox{for all }\eta =(\eta_1,\eta_2)\mbox{ such that }\frac{\eta_2}{\eta_1}\in V_\lambda.
\end{equation}
The \emph{$N$-Wavefront Set} $\mbox{WF}^N(f)$ is the complement of the set of $N$-regular directed points.
The \emph{Wavefront Set} $\mbox{WF}(f)$ is defined as $\bigcup_{N>0}\mbox{WF}^N(f)$.
\end{definition} 

It is not clear at first sight that the Wavefront Set according to the definition given above is well-defined, meaning that it is independent of
 the localization function $\Phi$. Below we prove that it actually is: To keep things simple we shall mostly restrict ourselves to $f\in L^2(\R^2)$.

\begin{lemma}[Localizing does not enlarge the Wavefront Set]\label{lem:loc}
Let $f\in L^2(\R^2)$ be a function such that 
$$
	\hat f (\eta) = O\big((1 + |\eta|)^{-N}\big)\quad \mbox{for all }\eta =(\eta_1,\eta_2)\mbox{ such that }\frac{\eta_2}{\eta_1}\in V_\lambda 
$$
where $V_\lambda$ is some open subset of $\R$ with $\lambda\in V_\lambda$. Let $\Phi$ be a smooth function. Then there exists an open neighbourhood $V'_\lambda$ of $\lambda \in \R$ such that
$$
	(\Phi f)^\land (\eta) = O\big((1 + |\eta|)^{-N}\big)\quad \mbox{for all }\eta =(\eta_1,\eta_2)\mbox{ such that }\frac{\eta_2}{\eta_1}\in 			V'_\lambda. 
$$
In other words: If $(x,\lambda)$ is an $N$-regular point of $f$, it is also an $N$-regular point of $\Phi f$.
\end{lemma}

\begin{proof}
Let $\lambda_0\in V_\lambda$ and $e_0$ a unit vector in $\R^2$ with slope $\lambda_0$. We want to show that
$$
	(\Phi f)^\land (te_0) = O(|t|^{-N}).
$$
Let us compute 
\begin{equation}\label{eq:con}
	(\Phi f)^\land (te_0) =\hat \Phi \ast \hat f (te_0) =\int_{\R^2}\hat f (te_0 - \xi)\hat \Phi(\xi)d\xi. 
\end{equation}
Since $\lambda_0\in V_\lambda$ and $V_\lambda$ is open, there exists $1>\delta >0$ such that $te_0 + tB_\delta$ is still contained in the cone of all points $\xi $ with $\xi_2/\xi_1 \in V_\lambda$ for all $t\in \R$ and $B_\delta$ the ball with radius $\delta$ around the origin.
Now we split the integral (\ref{eq:con}) into
$$
	(\ref{eq:con}) = \underbrace{\int_{|\xi| < \delta t}\hat f (te_0 - \xi)\hat \Phi(\xi)d\xi}_{A} + \underbrace{\int_{|\xi| > \delta t}\hat f 	(te_0 - \xi)\hat \Phi(\xi)d\xi}_{B}.
$$
By assumption we know that for $|\xi|<\delta t$ we can estimate
$$
	|\hat f (te_0 - \xi)| = O(|te_0 -\xi|^{-N}) = O(|t|^{-N}).
$$ 
Therefore it is easy to see that $|A| = O(|t|^{-N})$.
To estimate $B$ we employ Cauchy-Schwarz and again the smoothness of $\Phi$:
\begin{eqnarray*}
	|B| &\le &  \int_{|\xi| > \delta t}|\hat f (te_0 - \xi)| |\hat \Phi(\xi)|d\xi\le
	\|f\|_2\int_{|\xi| > \delta t}|\hat \Phi(\xi)|^2d\xi \\
	&\le & C\|f\|_2\int_{|\xi| > \delta t}(1+|\xi|)^{-N-1}d\xi = O(|t|^{-N}).
\end{eqnarray*}
\end{proof}

\begin{remark}\label{rem:loc}
While the above lemma assumes that $f\in L^2(\R^2)$, the result is actually valid for any tempered distribution $f$. This can be shown as above by first assuming that $\hat f$ is a slowly growing function and using the H\"older inequality instead of Cauchy-Schwartz. The case of a general tempered distrubution $f$ is then handled by repeated partial integration.
\end{remark}

%-------
\subsection{The Radon transform}
%-------
We introduce the Radon transform \cite{Deans1983}. As we shall see later it will serve us as a valuable tool in the proofs of the later sections.

\begin{definition}
The Radon transform of a function $f$ is defined by
\begin{equation}
	\mathcal{R} f(u,s) := \int_{x_2\in \R} f( u-sx_2 , x_2) dx_2.
\end{equation}
\end{definition}

Observe that our definition of the Radon transform differs from the most common one which parametrizes the directions
in terms of the angle and not the slope as we do. It turns out that our definition is particularily well-adapted to
the mathematical structure of the shearlet transform.
The next theorem already indicates that the Radon transform provides a useful tool in studying microlocal phenomena.

\begin{theorem}[Projection Slice Theorem]
\begin{equation}
	(\mathcal{R}f(u,s))^\land (\omega ) = \hat f (\omega ( 1,s)).
\end{equation}
\end{theorem}

For the convenience of the reader and because the proof is so short we show how this can be proved:

\begin{proof}
\begin{eqnarray*}
	(\mathcal{R}f(u,s))^\land (\omega ) &=& \int_\R\int_\R f(u-sx_2 , x_2) e^{-2\pi i u\omega}dx_2 du \\
	&=&\int_\R\int_\R f(\tilde u,x_2)e^{-2\pi i (\tilde u + sx_2)\omega}dx_2d\tilde u=\hat f(\omega(1,s)).
\end{eqnarray*}
\end{proof}

By the Projection Slice Theorem, another way of stating that $(x,\lambda)$ is an $N$-regular directed point is that
$$
	(\mathcal{R}\Phi f(u,s))^\land (\omega ) = O(|\omega |^{-N}) \quad \mbox{and }s \in V_\lambda, 
$$
or in other words, that 
$\mathcal{R}\Phi f(u,s)$ is sufficiently smooth in $u$ around $s= \lambda$.

Since wavelets can deal perfectly well with univariate functions it is a natural idea to perform a wavelet analysis
on the Radon transform of a bivariate function in order to study directional properties. This idea has been implemented
in the ridgelet transform \cite{Candes1999a} and has led to the construction of curvelets and shearlets as we now briefly explain: 

The original ridgelet transform first partitions the function  $f$ to be studied into parts which
are located in small but fixed spacial rectangles and then represents the Radon transform of each of these parts via a wavelet transform. It can be shown that
any $f$ with only directional singularities along a curve with small curvature can be effectively represented in this way. However, 
if there are no restrictions posed on the shape of the singular set of $f$, a multiscale approach is required which leads to multiscale
ridgelets \cite{Candes1999b}. It turns out that the whole set of multiscale ridgelets is too redundant and not frameable. A solution to
this problem was given with the 'first generation curvelet transform' which first applies a bandpass filter to $f$ and then analyzes the
the frequency band $\sim 2^j$ with multiscale ridgelets of scale $\sim 2^{j/2}$ -- corresponding to the parabolic scaling relation. 
Realizing that the crucial thing is the relation (\ref{eq:scalrel}) a much simpler curvelet construction has been given in \cite{candes2002} and this construction is what is now usually referred to as 'curvelets'. 
	
%---------------------------------------------------------------------	
\section{A direct theorem}\label{sec:direct}
%---------------------------------------------------------------------

In this section we show that for an $N$-regular directed point $(t_0,s_0)$ of a function $f\in L^2(\R^2)$ the
shearlet coefficients with respect to any function with sufficiently many vanishing moments in the $x_1$-direction decay quickly around $t=t_0,\ s=s_0$. This fact is especially important for applications where sparse representations of objects with edges are desired \cite{candes2002,guo2007}.  We assume here that $f$ is square integrable since for general tempered distributions we would have to require $\psi$ to be a Schwartz function. The extension to general tempered distributions is however straightforward, see Remark \ref{rem:loc}.

\begin{theorem}[Direct Theorem]\label{thm:direct}
Assume that $f$ is an $L^2(\R^2)$-function and that $(t_0,s_0)$ is an $N$-regular directed point of $f$. 
Let $\psi \in H_{(0,L)}(\R^2),\ \hat \psi \in L^1(\R^2)$ be a shearlet with $M$ moments which satisfies a decay estimate of the form
\begin{equation}\label{eq:decay2}
	\psi(x) = O\big((1 + |x|)^{-P}\big).
\end{equation} 
Then there exists a neighbourhood $U(t_0)$ of $t_0$ and $V(s_0)$ of $s$ such that for any $1/2<\alpha < 1$, $t\in U(t_0) $
and $s\in V(s_0)$ we have the decay estimate
\begin{equation}\label{eq:shdecay}
	\mathcal{SH}_\psi f (a,s,t) = O\big(a^{-3/4 + P/2}  + a^{(1-\alpha)M} + a^{-3/4 + \alpha N} + a^{(\alpha - 1/2)L}\big).
\end{equation}
\end{theorem}

\begin{proof}
First we show that we can without loss of generality assume that $f$ is already localized around $t_0$, i.e. 
$f=\Phi f$ where $\Phi$ is the cutoff function from the definition of the $N$-wavefront set which equals $1$ around $t_0$.
To show this we prove that
\begin{equation}\label{eq:loc}
	\langle (1-\Phi)f , \psi_{ast}\rangle = O(a^{-3/4 + P/2}).
\end{equation} 
By definition we have
\begin{equation}\label{eq:ast}
	\psi_{ast}(x_1,x_2) = a^{-3/4}\psi\big(\frac{(x_1 - t_1) + s(x_2 - t_2)}{a}, \frac{x_2 - t_2}{a^{1/2}}\big).
\end{equation}
Now we note that in computing the inner product (\ref{eq:loc}) we can assume that $|x - t| > \delta$ for 
some $\delta >0$ and $t$ in a small neighbourhood $U(t_0)$ of $t_0$ since $(1- \Phi ) f = 0$ around $t_0$. By (\ref{eq:decay2}) we estimate
\begin{eqnarray*}
	|\psi_{ast}(x)|&\le & Ca^{-3/4}(1+|\left(\begin{array}{cc}a^{-1}&sa^{-1}\\0 & a^{-1/2}\end{array}\right)(x-t)|)^{-P} \\
	& \le & Ca^{-3/4}(1 + \|\left(\begin{array}{cc}a&-sa^{1/2}\\0 & a^{1/2}\end{array}\right)\|^{-1}|x-t|)^{-P}\\
	&\le&
	Ca^{-3/4}(1 + C(s)a^{-1/2}|x-t|)^{-P} = O(a^{-3/4+P/2}|x-t|^{-P})
\end{eqnarray*}
for $|x-t|>\delta$
and $C(s) = (1 + \frac{s^2}{2}+(s^2 + \frac{s^2}{4})^{1/2})^{1/2}$ (compare \cite[Lemma 5.2]{Labate2009}).
We can now estimate
\begin{eqnarray}
	\langle (1-\Phi)f , \psi_{ast}\rangle &\le & Ca^{-3/4 +P/2}\int_{|x-t|\geq \delta}|x-t|^{-P}|1-\Phi(x_1,x_2)||f(x_1,x_2)|dx_1 dx_2 	\nonumber\\
	&=& O(a^{-3/4 + P/2})
	\label{eq:1}
\end{eqnarray}
for $t\in U(t_0)$ and this is (\ref{eq:loc}).
Now, assuming that $f=\Phi f$ is localized, we go on to estimate the shearlet coefficients $|\langle f , \psi_{ast}\rangle| $. First
note that the Fourier transform of $\psi_{ast}$ is given by 
\begin{equation}\label{eq:pastft}
	\hat\psi_{ast}(\xi) = a^{3/4}e^{-2\pi i t \xi}\hat \psi \big( a\xi_1 , a^{1/2}(\xi_2 - s\xi_1)\big).
\end{equation}
Now pick $\frac12 < \alpha < 1$ and write
\begin{eqnarray}
	|\langle f , \psi_{ast}\rangle| &=& |\langle \hat f , \hat \psi_{ast}\rangle|\le a^{3/4}\int_{\R^2} |\hat f (\xi_1,\xi_2)| |\hat \psi \big( 	a\xi_1 , a^{1/2}(\xi_2 - s\xi_1)\big)|d\xi \nonumber\\
	&=&\underbrace{a^{3/4}\int_{|\xi_1| < a^{-\alpha}}}_{A}+\underbrace{a^{3/4}\int_{|\xi_1| > a^{-\alpha}}}_{B}.
\end{eqnarray}

Since $\psi$ possesses $M$ moments in the $x_1$ direction which means that $\hat \psi (\xi_1 , \xi_2) = \xi_1^M \hat \theta (\xi_1,\xi_2)$
with some $\theta\in L^2(\R^2)$, we can estimate $A$ as
\begin{eqnarray}
	A&=&a^{3/4}\int_{|\xi_1|< a^{-\alpha}} |\hat f (\xi_1,\xi_2)| |\hat \psi \big( a\xi_1 , a^{1/2}(\xi_2 - s\xi_1)\big)|d\xi\nonumber \\
	&=& a^{3/4}\int_{|\xi_1|< a^{-\alpha}} a^M|\xi_1|^M|\hat f (\xi_1,\xi_2)| |\hat \theta \big( a\xi_1 , a^{1/2}(\xi_2 - s\xi_1)\big)|d\xi\nonumber\\
	&\le& a^{M(1-\alpha)}a^{3/4}\int_{|\xi_1|< a^{-\alpha}} |\hat f (\xi_1,\xi_2)| |\hat \theta \big( a\xi_1 , a^{1/2}(\xi_2 - s\xi_1)\big)|d\xi \nonumber\\\label{eq:estA}
	&\le&  a^{(1-\alpha)M}\langle |\hat f| ,|\hat \theta_{ast} |\rangle \le a^{(1-\alpha)M}\|\hat f\|_2 \|\hat \theta_{ast}\|_2 =a^{(1-\alpha)M}\|f\|_2\|\theta\|_2.
\end{eqnarray}
In order to estimate $B$ we make the following substitution:
\begin{equation*}
	\left(
	\begin{array}{cc}
		a & 0\\
		-a^{1/2}s & a^{1/2}
	\end{array}\right)
	\left(\begin{array}{c}
		\xi_1\\
		\xi_2\\
	\end{array}\right) =
	\left(\begin{array}{c}
		\tilde\xi_1\\
		\tilde\xi_2\\
	\end{array}\right),
	\quad
	d\xi_1 d\xi_2 = a^{-3/2}d\tilde\xi_1 d\tilde\xi_2.
\end{equation*}
Then 
\begin{equation}\label{eq:subst}
	B = a^{-3/4}\int_{\frac{|\tilde \xi_1|}{a} > a^{-\alpha}} |\hat f \big(\frac{\tilde \xi_1}{a} , \frac{s}{a}\tilde \xi_1 + a^{-1/2}\tilde 			\xi_2\big)|
	|\hat \psi \big(\tilde \xi_1 , \tilde \xi_2\big)|d\xi.
\end{equation}
Now we shall use that $(t_0,s_0)$ is a regular directed point of $f$. This means that there is a neighbourhood $(s_0 - \varepsilon , s_0 + \varepsilon)$ such that
\begin{equation}\label{eq:wff}
	\hat f (\eta_1 , \eta_2) \le C (1 + |\eta|)^{-N} \quad \mbox{for all }\frac{\eta_2}{\eta_1} \in (s_0 - \varepsilon , s_0 + \varepsilon).
\end{equation}
Looking at (\ref{eq:subst}) we now consider $\frac{\eta_2}{\eta_1}$ with $\eta_1 := \frac{\tilde \xi_1}{a}$, 
$\eta_2:=\frac{s}{a}\tilde \xi_1 + a^{-1/2}\tilde \xi_2$ and $\frac{\tilde\xi_1}{a} > a^{-\alpha}$ and get the estimate
\begin{equation}\label{eq:eta}
s - a^{\alpha - 1/2}\tilde \xi_2\le	\frac{\eta_2}{\eta_1} = s + a^{-1/2}\tilde \xi_2 \frac{a}{\tilde \xi_1} \le s + a^{\alpha - 1/2}\tilde \xi_2.
\end{equation}
By (\ref{eq:wff}) we have that 
\begin{equation}\label{eq:eta1}
	|\hat f \big(\frac{\tilde \xi_1}{a} , \frac{s}{a}\tilde \xi_1 + a^{-1/2}\tilde \xi_2\big)|\le C\big( 1 + \frac{|\tilde\xi_1|}{a}\big)^{-N}
\end{equation}
for $s$ in a neighbourhood $V(s_0)$ of $s_0$, $\frac{|\tilde \xi_1|}{a}>a^{-\alpha}$ and $|\tilde \xi_2| < \varepsilon' a^{1/2-\alpha}$ for 
some $\varepsilon'< \varepsilon$.
Now we first split the integral $B$ according to
\begin{eqnarray}
	B & = & a^{-3/4}\int_{|\tilde\xi_1|/a \geq a^{-\alpha}}|\hat f(\tilde \xi_1 / a , \frac{s}{a}\tilde \xi_1 + a^{-1/2}\tilde \xi_2)|| \hat \psi(\tilde\xi_1,\tilde\xi_2)|
	d\tilde\xi_1 d\tilde\xi_2 \nonumber \\
	&=& \underbrace{a^{-3/4}\int_{|\tilde\xi_1|/a \geq a^{-\alpha}, \ |\tilde\xi_2| < \varepsilon' a^{1/2-\alpha}}}_{B_1}
	+ \underbrace{a^{-3/4}\int_{|\tilde\xi_1|/a \geq a^{-\alpha}\ |\tilde\xi_2| > \varepsilon' a^{1/2-\alpha}}}_{B_2}
\end{eqnarray}
By (\ref{eq:eta1}) we can estimate $B_1$ according to
\begin{equation}\label{eq:est2}
	B_1 \le Ca^{\alpha N -3/4}\|\hat\psi\|_1
\end{equation}
It only remains to estimate $B_2$. For this we will use the fact that $\frac{\partial^L}{\partial x_2^L}\psi\in L^2(\R^2)$. This implies that
\begin{eqnarray}
	B_2 & \le & a^{-3/4}\int_{|\tilde\xi_1|/a \geq a^{-\alpha}\ |\tilde\xi_2| > \varepsilon' a^{1/2-\alpha}}|\hat f(\tilde \xi_1 / a , 
	\frac{s}{a}\tilde \xi_1 + a^{-1/2}\tilde \xi_2) \hat \psi(\tilde\xi_1,\tilde\xi_2)|
	d\tilde\xi_1 d\tilde\xi_2 \nonumber \\
	&=& a^{-3/4}\int_{|\tilde\xi_1|/a \geq a^{-\alpha}\ |\tilde\xi_2| > \varepsilon' a^{1/2-\alpha}}|\hat f(\tilde \xi_1 / a , 
	\frac{s}{a}\tilde \xi_1 + a^{-1/2}\tilde \xi_2) \tilde\xi_2^{-L}\big(\frac{\partial^L}{\partial 						x_2^L}\psi\big)^\land(\tilde\xi_1,\tilde\xi_2)|
	d\tilde\xi_1 d\tilde\xi_2 \nonumber  \\
	&\le & (\varepsilon')^{-L} a^{-3/4 + (\alpha - 1/2)L} \int_{\R^2}|\hat f(\tilde \xi_1 / a , 
	\frac{s}{a}\tilde \xi_1 + a^{-1/2}\tilde \xi_2)| |\big(\frac{\partial^L}{\partial 	x_2^L}\psi\big)^\land(\tilde\xi_1,\tilde\xi_2)|
	d\tilde\xi_1 d\tilde\xi_2\nonumber \\
	\label{eq:est3}
 	& = & (\varepsilon')^{-L} a^{(\alpha - 1/2)L}|\langle |\hat f| ,|( \frac{\partial^L}{\partial x_2^L}\psi_{ast})^\land|\rangle| \le  (\varepsilon')^{-L}
	a^{ (\alpha - 1/2)L}\|f\|_2 \|\frac{\partial^L}{\partial x_2^L}\psi\|_2.
\end{eqnarray}
Putting together the estimates (\ref{eq:1}), (\ref{eq:estA}), (\ref{eq:est2}) and (\ref{eq:est3}) we finally arrive at the desired conclusion.
\end{proof}

%--------------------------------------------------------------------------
\section{Frames}\label{sec:frame}
%--------------------------------------------------------------------------
The goal of this section is to extend the range of validity for the representation formula (\ref{eq:reprorig}) allowing general functions with vanishing moments. This
is important in order to enlarge the scope of potentially useful shearlets. These results will also serve as a tool to prove the main theorem in Section
\ref{sec:inverse}.
Note that the formula (\ref{eq:reprgroup}) is valid for any shearlet but this representation comes with some disadvantages such as dependence on the 
choice of coordinate axes. 
In what follows we present three different generalizations of (\ref{eq:reprorig}), each one with different advantages and disadvantages. 
In Theorem \ref{thm:frame1} we find a representation formula
which is valid for any choice of a window function $W$ but this representation is not \emph{tight} (to be defined later) and it requires to project the data $f$ onto a frequency cone prior to the analysis. Next, in Theorem \ref{thm:frame2} we eliminate the need to perform this projection at the cost of possibly worse frame constants. Finally, by carefully choosing the window function $W$ we show in Theorem \ref{thm:frame3} that we can actually wind up with tight frames if we perform a frequency projection onto a cone prior to the analysis.
First some definitions:
We utilize the concept of \emph{continuous frames} which has been introduced in \cite{Ali1993}. 

\begin{definition}Let $(X,\mu)$ be a measure space with Radon measure $\mu$ and $\Phi=(\varphi_x)_{x\in X}$ a family of elements in some Hilbert space
$\mathcal{H}$ indexed by $X$. $\Phi$ is called a \emph{frame} in $\mathcal{H}$ if there exist constants $A,B>0$ such that
\begin{equation}\label{eq:frame}
	A\|f\|_2^2 \le \int_{X}|\langle f,\varphi_x\rangle|^2 d\mu(x) \le B\|f\|_2^2\quad \mbox{for all }f\in \mathcal{H}.
\end{equation}
A frame is called \emph{tight} if $A=B$.
\end{definition}

If $\Phi$ is a frame for $\mathcal{H}$, we call the operator $\mathcal{T} : \mathcal{H} \to L^2(X,\mu), f \mapsto (\langle f , \varphi_x\rangle )_{x\in X}$ the \emph{analysis operator} and its adjoint $\mathcal{T}^*$ the \emph{synthesis operator}. The frame condition (\ref{eq:frame})
assures the boundedness and lower-boundedness of the \emph{frame operator} $\mathcal{S}: \mathcal{H}\to \mathcal{H}, f\mapsto \mathcal{T}^*\mathcal{T}f:= \int_{X}\langle f, \varphi_x\rangle \varphi_x d\mu(x)$.  
Call the family $\tilde \Phi=(\tilde \varphi_x)_{x\in X}$ with $\tilde \varphi_x := \mathcal{S}^{-1}\varphi_x$ the \emph{canonical dual frame} of $\Phi$. Then it is not hard to see that

\begin{itemize}
	\item[(i)] $\tilde \Phi$ is a frame with frame constants $\frac1B , \frac1A$, and
	\item[(ii)] We have the reproducing formula 
	\begin{equation}\label{eq:rep}
		f = \int_{x\in X}\langle f , \tilde \varphi_x\rangle \varphi_x d\mu(x) = \int_{x\in X}\langle f , \varphi_x\rangle\tilde \varphi_x 							d\mu(x), 
	\end{equation}
	where the equality holds at least in a weak sense.
	\item[(iii)]For tight frames with frame constant $A$ the dual frame elements are given by $\tilde\varphi_x=\frac{1}{A}\varphi_x$. 
\end{itemize} 
The 'quality criterion' for frames in terms of efficiency and accuracy of the inversion of the frame operator is the ratio $B/A$, so in general it is desirable for a frame
to be as close to tight as possible (such frames are also sometimes called 'snug frames'). 
An example of a frame for the Hilbert space $L^2(\R^2)$ would be the family $$\Psi=(\psi_{ast})_{(a,s,t)\in \R_+\times \R \times \R^2}$$ where
the measure $\mu$ has Lebesque density $a^{-3}$ and $\psi$ is a shearlet. In this case the frame constants are equal:
$A = B = C_\psi$ and the canonical dual frame $\tilde \Psi$ equals $\Psi$. 
The drawback of the representation of a function $f\in L^2(\R^2)$ arising from this frame is that the directions indexed by $s$ are not equally distributed as $s$ grows large. The goal of this section is therefore to restrict the parameters $a,s$ to a compact interval.
If the parameter $a$ is restricted to a compact interval this means that low frequency phenomena will not be representable. Similarly, if 
$s$ is restricted to a compact interval, this means that not all directions of singularities will be representable. We therefore shrink the Hilbert space $L^2(\R^2)$ to the smaller space 
$$
	L^2(\C_{u,v})^\lor:= \{ f\in L^2(\R^2):\ \mbox{supp }\hat f\subseteq \C_{u,v}\},
$$
where
$$ 
	\C_{u,v}:=\{\xi \in \R^2 : |\xi_1| \geq u ,\ |\xi_2 |\le v|\xi_1|\}.
$$
We shall always require $u,v>0$.
Once we succeed in constructing shearlet frames for the Hilbert space $L^2(\C_{1,1})^\lor$, we are able to represent any signal $f\in L^2(\R^2)$ as follows:
Let 
$$
	\C^\nu_{u,v} := \{\xi \in \R^2 : |\xi_2| \geq u ,\ |\xi_1 |\le v|\xi_2|\}.
$$
Denote by $P_{\C_{u,v}}$ the orthogonal projection onto the space $L^2(\C_{u,v})^\lor$ and by $P_{\C^\nu_{u,v}}$ the orthogonal projection
onto the space $L^2(C_{u,v}^\nu)^\lor$.
We write $P_\C:= P_{\C_{1,1}}$, $P_{\C^\nu}:=P_{\C_{1,1}^\nu}$ and $P_{\D}$ for the orthogonal projections onto the spaces
$L^2(\C_{1,1})^\lor$, $L^2(\C_{1,1}^\nu)^\lor$ and $L^2([-1,1]^2)^\lor$, respectively.
Then

\begin{itemize}
	\item $P_\C f$ is analyzed using the shearlet $\psi$,
	\item $P_{\C^\nu}f$ is analyzed using the shearlet $\psi^\nu(x_1,x_2):= \psi(x_2 , x_1)$, and
	\item $P_{\D}f$ is analyzed with some low pass window.
\end{itemize}

In view of detecting singularities, the function $P_{\D}f$ is not interesting since it is analytic. 
We therefore restrict our attention to the detection of directional features of the function
$P_\C f$ which we will henceforth simply denote by $f$. For the case of $P_{\C^\nu}f$ the analysis works analogous by reversing
the variables.

We will now analyze the structure of the frame operator on $L^2(\C_{u,v})^\lor$ related to a system $\Psi = (P_{\C_{u,v}}\psi_{ast})_{a\in [0,\Gamma],\ s\in [-\Xi , \Xi],\ t\in \R^2}$ and the usual measure $\mu$.  For notational simplicity we shall simply write
$\psi_{ast}$ for the functions $P_{\C_{u,v}}(\psi_{ast})$. In fact it makes little difference if we consider $\psi_{ast}$ since for any $f\in L^2(\C_{u,v})^\lor$ we have 
$$
	\langle f , \psi_{ast}\rangle = \langle P_{\C_{u,v}} f , \psi_{ast}\rangle = \langle f , P_{\C_{u,v}}\psi_{ast}\rangle.
$$
Due to the shift-invariance of the system $\Psi$, the frame operator possesses a particularly simple structure:

\begin{lemma}\label{lem:mult}
The frame operator $\mathcal{S}$ associated with the system $\Psi$ is a Fourier multiplier with the function
$$
	\Delta_{u,v}(\psi)(\xi):=\chi_{\C_{u,v}}(\xi)\int_{0< a < \Gamma ,\  |s|< \Xi} |\hat \psi \big(a\xi_1, \sqrt{a}(\xi_2 - s\xi_1)\big)|^2 a^{-3/2}da ds,
$$
$\chi_{\C_{u,v}}$ denoting the characteristic function of $\C_{u,v}$.
In particular, $\Psi$ is a frame for $L^2(\C_{u,v})^\lor$ if and only if there exist constants $A,B>0$ such that
$$
	A\le \Delta_{u,v}(\psi)(\xi) \le B\quad \mbox{ for all }\xi \in \C_{u,v}.
$$
\end{lemma}

\begin{proof}
Let $f,g\in L^2(\C)$. Then we compute

\begin{eqnarray*}
	&& \langle (\mathcal{S}f)^\land,\overline{\hat g}\rangle =\int_{\R^2}\int_0^{\Gamma} \int_{-\Xi}^{\Xi} \langle f , \psi_{ast}\rangle 	\overline{\langle \hat g , \hat\psi_{ast}\rangle} ds a^{-3}da dt\\
	&=&\int_{\R^2}\int_0^{\Gamma} \int_{-\Xi}^{\Xi} \langle \hat f , \hat\psi_{ast}\rangle\overline{ \langle \hat g ,\hat \psi_{ast}\rangle} ds 	a^{-3}dadt\\
	&=&\int_{\R^2}\int_0^{\Gamma} \int_{-\Xi}^{\Xi} \int_{\R^2}\int_{\R^2} \hat f(\xi) \hat\psi(a\xi_1,a^{1/2}(\xi_2 - s\xi_1))
	\overline{ \hat g(\eta) \hat \psi(a\eta_1,a^{1/2}(\eta_2 - s\eta_1))}e^{2\pi i t(\eta - \xi)} d\xi d\eta ds a^{-3/2}dadt \\
	&=& \int_0^{\Gamma} \int_{-\Xi}^{\Xi} \int_{\R^2}\hat f(\xi) \hat\psi(a\xi_1,a^{1/2}(\xi_2 - s\xi_1))
	\overline{ \hat g(\xi) \hat \psi(a\xi_1,a^{1/2}(\xi_2 - s\xi_1))} d\xi ds a^{-3/2}da \\
	&=& \int_{\R^2}\hat f(\xi) \overline{\hat g(\xi)}\Delta_{u,v}(\psi)(\xi) d\xi=\int_{\C}\hat f(\xi) \overline{\hat g(\xi)}\Delta_{u,v}(\psi)(\xi) d\xi.
\end{eqnarray*}
We have used the fact that $\int_{\R^2}e^{2\pi i t (\xi - \eta)} dt = \delta_{\eta = \xi}$ in the sense of oscillatory integrals, or in the other words the Fourier inversion formula.
Now it suffices to choose $\hat g $ to be some approximate identity in the convolution algebra $L^1(\R^2)$, e.g. Gaussian kernels to conclude
that 
$$
	(\mathcal{S} f)^\land (\omega) = \Delta_{u,v}(\psi)(\omega)\hat f (\omega).
$$
The rest is trivial.
\end{proof}
We can now show the important result that we can build frames from almost arbitrary functions with
vanishing moments in $x_1$-direction by letting the parameters $a$ and $t$ vary in a sufficiently large \emph{but finite} interval $[0,\Gamma]$, resp. $[-\Xi,\Xi]$. First some notational conventions:
Note that the Definition \ref{def:moments} makes also sense for general $n\in \R$. We
say that a function $\psi$ has Fourier decay of order $\varepsilon$ in the second variable if $\hat \psi(\xi_1,\xi_2)=O(|\xi_2|^{-\varepsilon})$ for large $\xi_2$.
Fourier decay in the first variable is defined in an analogous fashion.

\begin{theorem}\label{thm:frame}
let $\psi$ be a continuous shearlet with at least $1+\varepsilon>1$ directional moments, Fourier decay of order $\tau >1/2$ in the second coordinate and Fourier decay of order $\mu>0$ in the first variable. Then there exists $\Gamma, \Xi$ such that the family 
$(P_{\C_{u,v}} \psi_{ast})_{a\in [0,\Gamma],\ s \in [-\Xi,\Xi],\ t\in \R^2}$ constitutes a frame for $L^2(\C_{u,v})^\lor$.
\end{theorem}

\begin{proof}
The proof is given in the appendix.
\end{proof}
Using the previous results we can now show the following representation formula:

\begin{theorem}[Representation of $L^2$-functions, frequency projection]\label{thm:frame1}
Let $\psi$ be a shearlet such that $(P_{\C} \psi_{ast})_{a\in [0,\Gamma],\ s \in [-\Xi,\Xi],\ t\in \R^2}$ constitutes a frame for $L^2(\C_{1,1})^\lor$ with frame constants $A,B$ and let $W$ be any function 
with
\begin{equation}\label{eq:winframe}
	A \le |\hat W(\xi)|^2 \le B \quad \mbox{ for all } \xi \in [-1,1]^2.
\end{equation} Then we have the \emph{representation formula}
\begin{eqnarray}\nonumber 
	A\|f\|_2^2 &\le &\int_{\R^2}|\langle P_\D f , T_tW\rangle |^2dt + \int_{t\in \R^2}\int_{s\in [-\Xi,\Xi]}\int_{a\in 	[0,\Gamma]}|\mathcal{SH}_{\psi}P_{\C} f (a,s,t)|^2 a^{-3}da ds dt \\
	&& + \int_{t\in \R^2}\int_{s\in [-\Xi,\Xi]}\int_{a\in [0,\Gamma]}|\mathcal{SH}_{\psi^\nu} P_{\C^\nu} f (a,s,t)|^2 a^{-3}da ds dt
	\le B\|f\|_2^2
	\label{eq:repr}
\end{eqnarray}
for all $f\in L^2(\R^2)$.
In every point of continuity $x$ of $f$ we have
\begin{eqnarray}\nonumber
	f(x)&=&\int_{\R^2}\langle P_\D f , T_tW\rangle T_tP_\D \tilde W dt +
	\int_{t\in \R^2}\int_{s\in [-\Xi,\Xi]}\int_{a\in [0,\Gamma]}\mathcal{SH}_{\psi}P_{\C} f (a,s,t) P_\C\tilde\psi_{ast}(x) a^{-3}da ds dt \\
	&& +
	\int_{t\in \R^2}\int_{s\in [-\Xi,\Xi]}\int_{a\in [0,\Gamma]}\mathcal{SH}_{\psi^\nu} P_{\C^\nu} f (a,s,t) P_{\C^\nu}\tilde \psi_{ast}(x) 	a^{-3}da ds dt \label{eq:repr1}
\end{eqnarray}
where $\tilde W$ is any function with $(\tilde W(\xi))^\land =\hat W(\xi)^{-1}$ for all $\xi \in [-1,1]^2$. 
\end{theorem}

\begin{proof}
First we note that the frame operator for the system $(T_tW)_{t\in \R^2})$ for the Hilbert
space $L^2([-1,1]^2)^\lor$ is given by the Fourier multiplier with the function $\chi(\xi)|\hat W(\xi)|^2$,
where $\chi$ is the characteristic function of $[-1,1]^2$. The proof is standard. It follows that the dual 
frame is given by $(T_t\tilde W)_{t\in \R^2}$ and $\tilde W$ defined as above. It also follows from our assumptions that 
frame constants of the system $(T_tW)_{t\in \R^2})$ are given by $A,B$.
We already know that the systems $(P_{\C} \psi_{ast})_{a\in [0,\Gamma],\ s \in [-\Xi,\Xi],\ t\in \R^2}$ and
$(P_{\C^\nu} \psi^\nu_{ast})_{a\in [0,\Gamma],\ s \in [-\Xi,\Xi],\ t\in \R^2}$ constitute a frame
for $L^2(\C_{1,1})^\lor$, $L^2(\C^\nu_{1,1})^\lor$, respectively with frame constants $A,B$.
Since the three space $L^2(\C_{1,1})^\lor$, $L^2(\C^\nu_{1,1})^\lor$ and $L^2([-1,1]^2)^\lor$ are mutually orthogonal, (\ref{eq:repr}) follows. Equation (\ref{eq:repr1}) can be shown by polarization and convolution with an approximate identity (i.e. convolution with Gaussian kernels).
\end{proof}

Actually, there is no need to perform a frequency projection prior to the analysis and we
can wind up with a truly local procedure:

\begin{theorem}[Representation of $L^2$-functions, no frequency projection]\label{thm:frame2}
With the assumptions from Theorem \ref{thm:frame1} the system 
$$
	(\psi_{ast})_{a\in [0,\Gamma],\ s \in [-\Xi,\Xi],\ t\in \R^2}\cup (\psi^\nu_{ast})_{a\in [0,\Gamma],\ s \in [-\Xi,\Xi],\ t\in \R^2}\cup (T_tW)_{t\in \R^2}
$$ 
constitutes a frame for $L^2(\R^2)$ with frame constants $A, 3B$.
The associated frame operator is given by the Fourier multiplier with the function
$$
	\Omega(\xi):=\Delta_{0,\infty}(\psi)(\xi) + \Delta_{0,\infty}(\psi^\nu)(\xi) + |\hat W (\xi)|^2.
$$
We have the representation (valid in every point of continuity of $f$)
\begin{eqnarray}\nonumber
	f(x)&=&\int_{\R^2}\langle f , T_tW\rangle T_t\tilde W dt +
	\int_{t\in \R^2}\int_{s\in [-\Xi,\Xi]}\int_{a\in [0,\Gamma]}\mathcal{SH}_{\psi} f (a,s,t)\tilde\psi_{ast}(x) a^{-3}da ds dt \\
	&& +
	\int_{t\in \R^2}\int_{s\in [-\Xi,\Xi]}\int_{a\in [0,\Gamma]}\mathcal{SH}_{\psi^\nu}  f (a,s,t) \tilde \psi_{ast}(x) a^{-3}da ds dt, 
\end{eqnarray}

where $(\tilde \psi_{ast})^\land (\xi) =\frac{1}{ \Omega(\xi)}\hat\psi_{ast}(\xi),\ (\tilde \psi_{ast}^\nu)^\land(\xi) = \frac{1}{\Omega(\xi)}\hat\psi_{ast}^\nu(\xi) $ and $(\tilde W)^\land(\xi) = \frac{1}{\Omega(\xi)} \hat W(\xi)$.  
\end{theorem}

\begin{proof}
The fact that the frame operator is given by multiplication with $\Omega$ follows from Lemma \ref{lem:mult} and Theorem \ref{thm:frame1}. The estimate on the frame constants is immediate.
We remark that the estimate $3B$ for the upper frame constant is rather crude.
\end{proof}

\begin{remark}\label{rem:tight}
It seems to be a difficult (yet very important in our opinion) question if there exists a dual frame $(\tilde \psi_{ast})$ to the system
$(\psi_{ast})$ which carries the same structure, meaning that 
\begin{equation}\label{eq:dualstruct}
\tilde\psi_{ast}(x_1,x_2) = a^{-3/4}\tilde \psi\big(\frac{x_1-t_1 + s(x_2 - t_2)}{a},\frac{x_2-t_2}{a^{1/2}}\big)
\end{equation}
for some function $\tilde \psi$. 
Note however that at least by choosing $\Gamma$ and $\Xi$ large enough we can make the frame arbitarily close to tight.
One possible approach in order to get tight frames is to enforce the function $\Omega$ to be constant by choosing a suitable window function
$W$. This can always be done but the crux is to show that this $W$ is actually a useful window function in the sense that
it has e.g. fast Fourier decay. 
\end{remark}
The next lemma shows that if we restrict ourselves to data with frequency support in a cone, then the window functions that we get by enforcing
a tight frame property are actually useful. 
We plan to study this approach further in future work.
\begin{lemma}\label{lem:tightwin}
Define $W$ by
\begin{equation}\label{eq:tightwin}
	\Delta_{u,v}(\psi)(\xi) + |\hat W(\xi)|^2 = C_\psi \chi_{\C_{u,v}}(\xi).
\end{equation}
Assume that $\Xi > v,\ u\geq 0$ and that $\psi=\frac{\partial^M}{\partial x_1^M}\theta$ has $M$ anisotropic moments, Fourier decay of order $L_1$ in the first  variable
and that $\theta$ has Fourier decay of order $L_2$ in the second variable
such that
\begin{equation}\label{eq:tightcom}
	2M-1/2 > L_2>M>1/2.
\end{equation}
Then 
$$
	|\hat W(\xi)|^2 = O(|\xi|^{-2\min(L_1,L_2-M)}). 
$$
In particular if $\psi$ is sufficiently smooth and has sufficiently many vanishing moments then $W$ is a smooth function (i.e. a useful window function). 
\end{lemma}
\begin{proof}
By definition we have
\begin{eqnarray*}
	|\hat W(\xi)|^2 &=&  \chi_{\C_{u,v}}(\xi)\big(\int_{a\in\R,\  |s|>\Xi} |\hat \psi \big(a\xi_1, \sqrt{a}(\xi_2 - s\xi_1)\big)|^2 a^{-3/2}da ds \\
	&&+\int_{a>\Gamma,\ |s| < \Xi}|\hat \psi \big(a\xi_1, \sqrt{a}(\xi_2 - s\xi_1)\big)|^2 a^{-3/2}da ds\big).
\end{eqnarray*}
We start by estimating the second integral using the Fourier decay in the first variable:
\begin{eqnarray*}
	\int_{a>\Gamma,\ |s| < \Xi}|\hat \psi \big(a\xi_1, \sqrt{a}(\xi_2 - s\xi_1)\big)|^2 a^{-3/2}da ds
	&\le&2\Xi C \int_{a>\Gamma}(a|\xi_1|)^{-2L_1}a^{-3/2}da \\
	&=& O(|\xi_1|^{-2L_1})=O(|\xi|^{-2L_1})
\end{eqnarray*}
for all $\xi\in\C_{u,v}$.
To estimate the other term we need the moment condition and the decay in the second variable. We write $\hat \psi (\xi)=\xi_1^M\hat\theta(\xi)$
and $\xi=(\xi_1,r\xi_1),\ |r|<v$ for $\xi \in \C_{u,v}$.
We start with the high frequency part:
\begin{eqnarray*}
	\int_{a <1,\ |s| > \Xi}|\hat \psi \big(a\xi_1, \sqrt{a}(\xi_2 - s\xi_1)\big)|^2 a^{-3/2}da ds
	&=&\int_{a <1,\ |s| > \Xi}|a\xi_1|^{2M}|\hat \theta \big(a\xi_1, \sqrt{a}(\xi_2 - s\xi_1)\big)|^2 a^{-3/2}da ds \\
	&\le&C\int_{a <1,\ |s| > \Xi}|a\xi_1|^{2M}|\sqrt{a}(\xi_2 - s\xi_1)|^{-2L_2} a^{-3/2}da ds \\
	&=&C\int_{a <1,\ |s| > \Xi}a^{2M-L_2-3/2}|\xi_1|^{2M-2L_2}|r-s|^{-2L_2}da ds 
\end{eqnarray*}
we have used that $|r-s|$ is always strictly away from zero because $v<\Xi$.
By assumption $L_2=2M-1/2 -\varepsilon$ for some $\varepsilon >0$. 
Hence we can estimate further
\begin{eqnarray*}
	\dots &=&C|\xi_1|^{-2(L_2-M)}\int_{a <1,\ |s| > \Xi}a^{-1+\varepsilon}|r-s|^{-2L_2}da ds=O(|\xi|^{-2(L_2-M)}).
\end{eqnarray*}
The low-frequency part can simply be estimated as follows:
\begin{eqnarray*}
	\int_{a >1,\ |s| > \Xi}|\hat \psi \big(a\xi_1, \sqrt{a}(\xi_2 - s\xi_1)\big)|^2 a^{-3/2}da ds
	&\le& C|\xi_1|^{-2L_2}\int_{a >1,\ |s| > \Xi}a^{-3/2-L_2}|r-s|^{-2L_2}dads\\
	&=&O(|\xi|^{-2L_2}).
\end{eqnarray*}
Putting these estimates together proves the statement.
\end{proof}
\begin{remark}
We find it quite interesting how the smoothness and the moment conditions have to interact in the shearlet transform. This stands in contrast to the wavelet transform where
only smoothness is required to arrive at a statement similar to Lemma \ref{lem:tightwin}. 
Similarily, for wavelets to satisfy a direct theorem they are only required to satisfy moment conditions and essentially no smoothness. Shearlets, on the other hand, need to satisfy moment- and smoothness conditions. Of course, when we speak of smoothness related to wavelets, we mean conventional smoothness and not directional smoothness as measured via the Wavefront Set.
\end{remark}
Using the previous lemma we are finally able to obtain tight frames for $L^2(\C_{u,v})^\lor$.
\begin{theorem}[Representation of $L^2(\C_{u,v})^\lor$-functions, tight, frequency projection]\label{thm:frame3}
With the assumptions of Lemma \ref{lem:tightwin} and $W$ defined as in (\ref{eq:tightwin}), the system
$$
	(P_{\C_{u,v}}\psi_{ast})_{a\in [0,\Gamma],\ s \in [-\Xi,\Xi],\ t\in \R^2}\cup (T_tP_{\C_{u,v}}W)_{t\in \R^2}
$$ 
constitutes a tight frame for $L^2(\C_{u,v})$ with frame constant $C_\psi$.
We have the representation
\begin{eqnarray}\nonumber
	f(x)&=&\frac{1}{C_\psi}\int_{\R^2}\langle f , T_tW\rangle T_tP_{\C_{u,v}}Wdt \\
	&&+\frac{1}{C_\psi}\int_{t\in \R^2}\int_{s\in [-\Xi,\Xi]}\int_{a\in [0,\Gamma]}\mathcal{SH}_{\psi} f (a,s,t)P_{\C_{u,v}}\psi_{ast}(x) a^{-3}da ds dt.
\end{eqnarray}
The window function $W$ satisfies the Fourier decay estimate from Lemma \ref{lem:tightwin}.
\end{theorem}
\begin{proof}
The frame operator is given as the Fourier multiplier with the function $\Delta_{u,v}(\psi)(\xi) + \chi_{\C_{u,v}}(\xi)|\hat W(\xi)|^2=\chi_{\C_{u,v}}(\xi)C_\psi$.
It follows that the frame operator is given by $C_\psi P_{\C_{u,v}}=C_\psi I$ on $L^2(\C_{u,v})$ ($I$ being the identity).
\end{proof}

\begin{remark}
For applications it is important to restrict the parameters $a,s,t$ to a discrete set. This problem has already been studied in the general framework of continuous frames in \cite{Fournasier2005}.  Let us remark however that we see little hope that discretizing continuous frames will lead to discrete and non-bandlimited \emph{tight} frames.
The reason for this pessimism is that we are not aware of any useful tight frame construction for non-bandlimited wavelets which comes from discretizing the continuous wavelet transform. To our knowledge the only useful and general method to construct wavelet tight frames (or wavelet frames with wavelet duals) comes from Multiresolution Analysis constructions in combination with
the 'unitary extension principle' \cite{ron1997}. We are currently pursuing the goal to generalize this construction to the shearlet setup \cite{Grohs10}.
\end{remark}

%---------------------------------------------------------------------------
\section{An inverse theorem}\label{sec:inverse}
%---------------------------------------------------------------------------
In this section we prove a partial converse to Theorem \ref{thm:direct}. We show that 
if the shearlet coefficients of a function around $(t_0,s_0)$ decay sufficiently fast in $a$, then $(t_0,s_0)$ is a regular directed point. Before we can get to the proof we need some localization results. First we show that a frequency projection on a conical set retains the wavefront set.

\begin{lemma}\label{lem:proj}
The point $(t_0,s_0)$ is an $N$ - regular directed point of $g\in L^2(\R^2)$ if and only if $(t_0,s_0)$ is an $N$ - regular directed point 
of $P_{\C_{u,v}} g$ provided that $s_0 < v$.
\end{lemma}

\begin{proof}
This statement is not trivial as it might seem at first glance. We first show the 'only if'-part. Write $P_\C g =g - P_{\C_{u,v}^c} g$, where $P_{\C_{u,v}^c}$ is the orthogonal projection onto the (closure of the) complement of $\C_{u,v}$. By assumption there exists a cutoff function $\Phi$ supported around $t_0$ such that
$$
	(\Phi g )^\land (\xi) = O(|\xi|^{-N})\quad \mbox{for all }\xi_2/ \xi_1 \in (s_0 -\delta ,s_0 + \delta)
$$ 
for some $\delta >0$. Clearly, since $s_0\in (-v,v)$ the point $(t_0 , s_0)$ is an $N$ - regular point of 
$P_{\C^c_{u,v}} g$. Therefore the same estimate as above also holds for $(P_{\C^c_{u,v}} g)^\land$.
The point is now that by Lemma \ref{lem:loc} the same estimate holds also for $( \Phi P_{\C^c_{u,v}} g)^\land$.
It follows that an analogous estimate holds for $(\Phi P_{\C_{u,v}} g)^\land$.
This proves the 'only if'-part. For the proof of the 'if'-part we estimate $( \Phi P_{\C^c_{u,v}} g)^\land$ with the same method as 
in the proof of Lemma \ref{lem:loc} and see that it is negligible for the decay properties of $(\Phi g)^\land$ restricted
to a small cone around the line with slope $s_0$.
\end{proof}

\begin{remark}
We do not know if the above lemma still holds true for $s_0 = v$. 
\end{remark}

We recall the definition of the $N$-th \emph{fractional derivative} of a function $f$ defined by
$$
	I^{(N)}(u) :=\big(\frac{\partial}{\partial u}\big)^N I (u):= \big(\omega^N \hat I(\omega)\big)^\lor (u).
$$
The following lemma states some well-known results for fractional derivatives with $N\notin \N$.

\begin{lemma}\label{lem:frac}
\begin{equation}\label{eq:fracchain} \big( f(\cdot /a) \big)^{(N)}(x) = a^{-N} f^{(N)} (x/a).\end{equation}
\begin{equation}\label{eq:fracprod} \|\big(fg)^{(N)}\|_2\le
C(\|f^{(N)}\|_4\|g\|_4 + \|f\|_4\|g^{(N)}\|_4) \quad N<1.
\end{equation}
\end{lemma}

\begin{proof}
The first equation (\ref{eq:fracchain}) is an easy exercise while equation (\ref{eq:fracprod})
is a special case of the so-called "fractional product rule", cf. \cite{F1991}.
\end{proof}

The following lemma shows that in studying the regularity properties of $f$ around $t_0$ only the shearlet coefficients of $f$ around $t_0$ are relevant. We only formulate and prove it for $N\in \N$, the general case can
be shown using Lemma \ref{lem:frac}. 

\begin{lemma}\label{lem:coeffloc}
Let $f\in L^2(\R^2)$, $\Phi$ be a smooth bump function supported in a small neighbourhood $V(t_0)$ of
some $t_0\in\R^2$ and let $U(t_0)$ be another neighbourhood of $t_0$ with $V(t_0)+\delta B\subset U(t_0)$ for some $\delta >0$. Here, $B$ denotes the unit disc in $\R^2$ and $+$ denotes
the Minkowski sum of two sets.
Consider the function 
\begin{equation}
	g(x) = \int_{t\in U(t_0)^c, s\in [-\Xi,\Xi], a \in [0,\Gamma]} \langle f  , \tilde \psi_{ast}\rangle
	\Phi(x)\psi_{ast}(x)a^{-3}da ds dt.
\end{equation}
Then for all $u,v$
\begin{equation}\label{eq:locdec}
	\hat g(\xi)  = O(|\xi|^{-N})\quad |\xi|\to\infty,\quad \mbox{for }\xi \in \C_{u,v}
\end{equation}
 provided that
\begin{equation}\label{eq:loc2}
	\theta^j(x):=\big(\frac{\partial}{\partial x_1}\big)^{j}\psi(x) = O(|x|^{-P_j})\quad \mbox{ with }P_j/2 -3/4> j + 2,\ j=0,\dots , N.
\end{equation}
\end{lemma}

\begin{proof}
Consider the Radon transform 
$$
	I(u):= \mathcal{R}g(u,s) = \int_\R g(u - sx_2 , x_2)dx_2,\quad |s|\le v.
$$
We will show that $I^{(N)}\in L^1(\R)$ which implies that 
$$
	(I^{(N)})^\land(\omega) = \omega^N \hat I(\omega) \in L^\infty(\R)
$$
which by the projection slice theorem implies that with $\xi = \omega(1,s)$
$$
	|\hat g(\xi)| = |\hat I(\omega)|\le \|I^{(N)}\|_{L_1(\R)}|\omega|^{-N} \le \|I^{(N)}\|_{L_1(\R)}\sqrt{1+s^2}|\xi|^{-N}
$$
which proves the statement.
We now show that $I^{(N)} \in L^1(\R)$.
In fact, since $I$ is of compact support we only need to show that $I^{(N)}$ is bounded.
\begin{eqnarray}\nonumber
	I^{(N)}(u) &=& \int_{ast}\langle f , \tilde \psi_{ast}\rangle \big(\frac{\partial}{\partial u}\big)^{N}\int_\R
	\Phi(u-sx,x)\psi_{ast}(u-sx,x)dx d\mu(ast)\\ \nonumber
	&=&
	\sum_{j=0}^N \binom{N}{j} \int_{ast}\langle f , \tilde \psi_{ast}\rangle \int_\R
	\big(\frac{\partial}{\partial u}\big)^{N-j}\Phi(u-sx,x)\big(\frac{\partial}{\partial u}\big)^{j}\psi_{ast}(u-sx,x)dx d\mu(ast)\\ 			\label{eq:loc1}
	&=&
	\sum_{j=0}^N \binom{N}{j} \int_{ast}\langle f , \tilde \psi_{ast}\rangle a^{-j}\int_\R
	\big(\frac{\partial}{\partial x_1}\big)^{N-j}\Phi(u-sx,x)\theta^j_{ast}(u-sx,x)dx d\mu(ast),
\end{eqnarray}
where $\theta^j = \big(\frac{\partial}{\partial x_1}\big)^{j}\psi$.
Argueing as in the localization part at the beginning of the proof of Theorem \ref{thm:direct}, we can show
that $|\theta^j_{ast}(x)|= O(a^{-3/4 +P_j/2}|x-t|^{-P_j})$. Since $\big(\frac{\partial}{\partial x_1}\big)^{N-j}\Phi \theta^j$ has small support around $t_0$ and the parameter
$t$ varies in a set which stays away from the support $V(t_0)$ of $\big(\frac{\partial}{\partial x_1}\big)^{N-j}\Phi\theta^j$
we can estimate
$$
	|\big(\frac{\partial}{\partial x_1}\big)^{N-j}\Phi(x)\theta^j_{ast}(x)|= O(|\big(\frac{\partial}{\partial x_1}\big)^{N-j}\Phi(x)|a^{-3/4 +P_j/2}|t-t_0|^{-P_j})
$$
for $t\in U(t_0)^c$.
Putting this estimate into (\ref{eq:loc1}) and using (\ref{eq:loc2}) we arrive at the desired statement.
\end{proof}

We are finally in a position to prove the main result of this section, namely an inverse theorem. Again, we
only formulate and prove it for $N\in \N$, the extension to arbitrary $N$ can be achieved via Lemma \ref{lem:frac}.

\begin{theorem}[Inverse Theorem] \label{thm:inverse}Let $f\in L^2(\C_{u,v})^\lor$, $\infty > u,v>0$.
Assume that there exist neighborhoods $U(t_0)\subset \R^2$ of $t_0$ and $(s_0 - \varepsilon , s_0 + \varepsilon)\subset [-s_0 , s_0]$ of $s_0$ such that 
\begin{equation}\label{eq:invass}
	\mathcal{SH}_\psi f (a , s, t) = O(a^{K}) \quad \mbox{for all } ( s,t)\in (s_0 - \varepsilon , s_0 + \varepsilon)\times U(t_0)
\end{equation} 
with the implied constant uniform over $s$ and $t$. Then $(t_0,s_0)$ is an $N$- regular directed point of $f$
for all $N$ with (\ref{eq:loc2}) such that $\psi\in H_{(N,L)}(\R^2)$, $\hat\theta^j$,  $\omega_1^{-M}\hat\psi(\omega)$, $\big(\frac{\partial^L}{\partial x_2^L}\theta^j\big)^\land \in L^1(\R^2),\ j=0,\dots , N$ 
and for some $1/2 < \alpha < 1$
\begin{equation}\label{eq:conditions}
	N+2 < \min \big( K-3/4,(1-\alpha)(M+N)-3/4, (\alpha - 1/2)L-3/4,2(L_2-M+1),2(L_1+1)),
\end{equation}
where $M>1$ is the number of anisotropic moments of $\psi$, $L$ is the Fourier decay of $\psi$ in the second coordinate and
$L_1, L_2 $ are defined as in Lemma \ref{lem:tightwin} such that (\ref{eq:tightcom}) holds.
\end{theorem}

\begin{proof}
Choose $\Gamma , \Xi$ such that the system 
$$
	(P_{\C_{u,v+\kappa}} \psi_{ast})_{a\in [0,\Gamma],\ s \in [-\Xi,\Xi],\ t\in \R^2}\cup (T_tP_{\C_{u,v+\kappa}} W)_{t\in \R^2}
$$
constitutes a tight frame for $L^2(\C_{u,v+\kappa})$ and $v+\kappa >s_0$, with $W$ chosen according to Lemma \ref{lem:tightwin}.
The goal is to prove that for a localized version $\tilde f$ of 
\begin{equation}\label{eq:recinv}
	g = \int_{t\in \R^2 , \ s\in (-\Xi,\Xi) ,\ a\in (0,\Gamma)} \langle f , \psi_{ast}\rangle \psi_{ast} a^{-3}da ds dt
\end{equation}
 around $t_0$ the Fourier transform
of the function $I(u):=\mathcal{R}\tilde f (u,s_0)$ decays of order $|\omega|^{-N}$ for $|\omega |\to \infty$.
This would prove (by the projection slice theorem) that $(t_0,s_0)$ is a regular directed point of $g$.
To show that this already implies that $(t_0,s_0)$ is a regular directed point of $f$, we argue as follows:
By Theorem \ref{thm:frame3} we have the representation 
$$
	f = \frac{1}{C_\psi}P_{\C_{u,v+\kappa}}(g + \int_{t\in \R^2}\langle f,T_tW\rangle T_tWdt)
$$
It follows from Lemma \ref{lem:proj} that $(t_0,s_0)$ is an $N$-regular directed point of $f$
if it is an $N$-regular directed point of $g + \int_{t\in \R^2}\langle f,T_tW\rangle T_tWdt$.
By Lemma \ref{lem:tightwin} and (\ref{eq:conditions}), $(t_0,s_0)$ is an $N$-regular point of  $\int_{t\in \R^2}\langle f,T_tW\rangle T_tWdt$,
and therefore we only need to verify regularity for $g$, which we will now do.

First note that by Lemma \ref{lem:coeffloc} we can without loss of generality restrict the parameter $t$ in the integral (\ref{eq:recinv})
to  $U(t_0)$ if we multiply by a suitable cutoff function $\Phi$.
Therefore we need to study the regularity properties of 
$$
	\tilde f =\int_{t\in U(t_0) , \ s\in (-\Xi,\Xi) ,\ a\in (0,\Gamma)} \langle f , \psi_{ast}\rangle \Phi\psi_{ast} a^{-3}da ds dt,
$$
where $\Phi$ is supported in a small neighbourhood $V_0(t_0)\subset \subset U(t_0)$ around $t_0$.
Let us denote by $I(u)$ the function 
$$ 
	I(u):=\mathcal{R}\tilde f (u,s_0) 
$$
with
$$
	\mathcal{R}\tilde f (u,s_0)=\int_{t\in U(t_0) , \ s\in (-\Xi,\Xi) ,\ a\in (0,\Gamma)} \langle f , \psi_{ast}\rangle \mathcal{R}\Phi\psi_{ast}(u,s_0) a^{-3}da ds dt,
$$
and
$$
	\mathcal{R}\Phi\psi_{ast}(u,s_0) = a^{-3/4}\int_\R\Phi(u-s_0x_2,x_2) \psi\big( \frac{u - s_0 x_2 - t_1 + s(x_2 - t_2)}{a} , 	\frac{x_2}{a^{1/2}}\big) d x_2.
$$
To prove our goal that $\hat I(\omega) = O(|\omega |^{-N})$ we need to show that $\omega^N \hat I(\omega) \in L^\infty(\R)$ or the stronger statement that the fractional derivative $I^{(N)}$ of $I$ defined by
$$
	I^{(N)}(u) :=\big(\frac{\partial}{\partial u}\big)^N I (u):= \big(\omega^N \hat I(\omega)\big)^\lor (u)
$$
is in $L^1(\R)$. 

Unless stated otherwise in what follows the variables $a,s,t$ are allowed to vary over the sets
$[0,\Gamma]$, $[-\Xi,\Xi]$, $U(t_0)$, respectively.

By the usual product rule and the definition of $\mathcal{R}$ the quantity $\|I^{(N)}\|_1$ can be estimated by

$$
	C\max_{j=0,\dots,N}\int_\R\int_{a,s,t}| \langle f , \psi_{ast}\rangle| a^{-j}|\mathcal{R}\big(\frac{\partial^{N-j}}{\partial 		x_1^{N-j}}\Phi \frac{\partial^j}{\partial x_1^j}\psi_{ast}\big)(u,s_0)|d\mu(a,s,t)du
$$
We only treat the case $j=N$, the other cases can be dealt with analogously.

Assume that the function $\Phi$ is zero outside a cube of sidelength $\eta$ around $t_0$.
Then it is easy to see that the support of $I^N$ is contained in the interval
$I_U = [(t_0)_1 -s_0(t_0)_2 -2\eta , (t_0)_1 -s_0(t_0)_2 -2\eta]$ and the integration variable 
$x_2$ from the definition of $\mathcal{R}$ can be restricted to the interval
$I_X=[(t_0)_2-\eta,(t_0)_2+\eta]$, where $t_0 = ((t_0)_1,(t_0)_2)$.
We now separate the quantity 
$$
	\int_{I_U}\int_{a,s,t}| \langle f ,  \psi_{ast}\rangle| a^{-N}|\mathcal{R}\big(\Phi \frac{\partial^N}{\partial 	x_1^N}\psi_{ast}\big)(u,s_0)|d\mu(a,s,t)du =:A+B
$$
with 
$$
	A=\int_{I_U}\int_{a,s\in (s_0-\varepsilon,s_0+\varepsilon),t}| \langle f , \psi_{ast}\rangle| a^{-N}|\mathcal{R}\big(\Phi 	\frac{\partial^N}{\partial x_1^N}\psi_{ast}\big)(u,s_0)|d		\mu(a,s,t)du
$$
and
$$
	B=\int_{I_U}\int_{a,s\in [-\Xi,\Xi]\setminus (s_0-\varepsilon,s_0+\varepsilon),t}| \langle f ,  \psi_{ast}\rangle| 	a^{-N}|\mathcal{R}\big(\Phi \frac{\partial^N}{\partial x_1^N}		\psi_{ast}\big)(u,s_0)|d\mu(a,s,t)du.
$$
In order to estimate $B$ we note that 
$$
	\mathcal{R}\big(\Phi \frac{\partial^N}{\partial x_1^N}\psi_{ast}\big)(u,s_0)=\int_{I_X} \Phi\frac{\partial}{\partial x_1^N}\psi_{ast}(u-s_0x_2,x_2) d x_2 = \mathcal{SH}_\theta \Phi	\delta_{x_1 + s_0x_2 - u}(a,s,t),
$$
which means the shearlet transform of the delta distribution concentrated on the line 
$x_1 + s_0x_2 - u=0$ and localized by $\Phi$ w.r.p. to the shearlet $\theta = \frac{\partial^{N}}{\partial x_1^N}\psi$ with
$M+N$ moments.
It is well known and easy to see that for $s\in (-\Xi,\Xi)\setminus [s_0 - \varepsilon, s_0 + \varepsilon]$, the point $(t,s)$ is an $R$-regular directed point of $\delta_{x_1 + s_0x_2 - u}$ for all $R\in \mathbb{N}$ (in other words $(t,s)$ is in the \emph{analytic Wavefront Set of $\delta_{x_1 + s_0x_2 - u}$}), hence of $\Phi\delta_{x_1 + s_0x_2 - u}$ by Remark \ref{rem:loc}.
By using the same arguments as in the proof of Theorem \ref{thm:direct} we see that
for any $1/2<\alpha <1$ the estimate
\begin{equation}\label{eq:diracdec}
	\mathcal{SH}_\theta\Phi \delta_{x_1 + s_0x_2 - u}(a,s,t) = O\big(a^{(1-\alpha)(M+N) -3/4}+ a^{-3/4 + (\alpha - 1/2)L}\big)
\end{equation}
holds with the implied constant uniform over 
$t\in U(t_0) , \ s\in (-\Xi,\Xi)\setminus [s_0 - \varepsilon, s_0 + \varepsilon]$. The details are given in the appendix.
Since by assumption there exists $1/2<\alpha<1$ such that 
$$
	N+2 < \min ((1-\alpha)(M+N)-3/4, (\alpha - 1/2)L-3/4),
$$
the expression $B$ is bounded.

In order to estimate $A$ we use the fast decay of the shearlet coefficients of $f$ around $(t_0,s_0)$. By our assumptions on $N,\ K$, the coefficients 
$\langle f ,  \psi_{ast}\rangle$ decay of order greater than
$a^{N+2+3/4}$, and therefore $A$ is bounded.
\end{proof}

Following this rather technical theorem we state an informal version of the inverse theorem:

\begin{theorem}[inverse theorem, informal version]\label{thm:invinf}
Assume that $\psi$ has sufficiently many vanishing moments in the $x_1$ direction, is sufficiently smooth and sufficiently well-localized in space.
Assume that (\ref{eq:invass}) holds for some $f$. Then $(t_0,s_0)$ is an $N$-regular directed point of $f$ for all $N< K-11/4$.
\end{theorem} 

%------------------------------------------------------------------------------------------
\section{Resolution of the Wavefront Set}\label{sec:res}
%------------------------------------------------------------------------------------------
We now draw some conclusions to the previous results. They all follow immediately from Theorems \ref{thm:direct} and \ref{thm:inverse}.
In terms of resolving the $N$-Wavefront Set we have the following result which we formulate only for $f\in L^2(\R^2)$ for simplicity.

\begin{theorem}[Resolution of the Wavefront Set I]
Let $f\in L^2(\R^2)$, $N\in \R$ and $\varepsilon >0$. Then there exist $P,M,L,L_1,L_2$ such that for all functions $\psi\in H_{(N,0)}(\R^2)$ with $M$ vanishing moments in $x_1$-direction, decay of order $P$ towards infinity, $C^L$ in the second coordinate and $L_1,L_2$ as in Lemma \ref{lem:tightwin} we have the following result: write $\mathcal{D}=\mathcal{D}_1\cup\mathcal{D}_2$, where
$\mathcal{D}_1 = \{(t_0,s_0)\in \R^2 \times [-1,1]:$ for $(s,t)$ in a neighbourhood $U$ of $(s_0,t_0)$, 
$|\mathcal{SH}_\psi f(a,s,t)| = O(a^N)$, with the implied constant uniform over $U\}$ and 
$\mathcal{D}_2 = \{(t_0,s_0)\in \R^2 \times (1,\infty]:$ for $(1/s,t)$ in a neighbourhood $U$ of $(s_0,t_0)$, 
$|\mathcal{SH}_{\psi^\nu}f(a,s,t)| = O(a^N)$, with the implied constant uniform over $U\}$.
Then 
\begin{equation}\label{eq:rwf}
	\mbox{WF}^{N+3/4+\varepsilon}(f)^c \subseteq  \mathcal{D}\subseteq 		\mbox{WF}^{N-11/4-\varepsilon}(f)^c.
\end{equation}
The precise values of $P,M,L,L_1,L_2$ can be read off Theorems \ref{thm:direct} and \ref{thm:inverse}.
\end{theorem}

\begin{proof}
First we need to show that if $|\mathcal{SH}_\psi f(a,s,t)| = O(a^N)$ for $|s|\le 1$, then $|\mathcal{SH}_\psi P_{C_{u,v}}f(a,s,t)| = O(a^N)$ with some suitable cone with $v>1$ and $P,L,M$ large enough. To this end we can estimate the integral $\int_{\C_{u,v}^c}\hat f(\xi)\overline{ \hat \psi_{ast}(\xi)}d\xi = O(a^N)$ using precisely the same estimates as in the proof of Theorem \ref{thm:direct} except for the estimate for $B_1$ in that proof. This shows that 
indeed $|\mathcal{SH}_\psi P_{C_{u,v}}f(a,s,t)| = O(a^N)$. The reverse implication can also be shown using the same argument. We arrive at the following statement:
$$|\mathcal{SH}_\psi f(a,s,t)| = O(a^N)\mbox{ for }|s|\le 1 \Leftrightarrow |\mathcal{SH}_\psi P_{C_{u,v}}f(a,s,t)| = O(a^N)$$ for some cone with $v>1$ and $P,L,M$ large enough.
The case $s>1$ is similar. We also need the fact that for $s\le 1$ the point $(t,s)$ is an $N$-regular point of $P_{\C_{u,v}}f$ if and only if $(t,s)$ is an $N$-regular point of $f$which is Lemma \ref{lem:proj}.
Now the statement follows directly from Theorems \ref{thm:direct} and \ref{thm:invinf}.
\end{proof}

\begin{remark}
Certainly it would be desirable to have an equality in (\ref{eq:rwf}) instead of the inclusions that we obtained. Despite considerable effort we were not able to obtain such a result and we are not sure if such a results holds at all. We believe that the reason for this is that our notion of Wavefront Set does not correspond to any useful microlocal function space. Usually Fourier decay of a function is not measured as in our definition of the $N$-Wavefront Set but rather in terms of a Sobolev (or more generally Besov) norm restricted to a cone like for instance 
$$
	\int_{\C_\lambda} (1+|\xi|^2)^N |\hat f (\xi)|^2  <\infty
$$
for all $\xi$ in a cone $\C_\lambda$ around the direction $\lambda$. Such measurements of the directional Fourier decay of a (localized version of a) tempered distrubution lead to the concept of \emph{microlocal Sobolev spaces}. In \cite{candes2003} microlocal Sobolev regularity has been fully characterized in terms of a curvelet square function for tempered distributions (assuming compact frequency support for the curvelets). In future work we would like to generalize these results to our setting.
However, for our present purpose, which is essentially to generalize the results in \cite{Labate2009}, the results of the previous sections are --  as we shall see below in Theorem \ref{thm:rwf2} -- just what we need.
\end{remark}

In \cite{Labate2009}, the authors considered the full Wavefront Set defined as
$$
	\mbox{WF}(f)= \bigcup_{N\in\R} \mbox{WF}^N( f)
$$
and showed that for very specific choices of $\psi$ the decay rate of the shearlet transform coefficients determine $\mbox{WF}f$ where $f$ is a tempered distribution. We show that the result is actually valid for any
Schwartz functoin with infinitely many vanishing moments in $x_1$-direction.
First we show this for $f$ with frequency support in a conical wedge.

\begin{theorem}
Assume that $\psi$ is a Schwartz test function with infinitely many vanishing moments in $x_1$-direction. Then 
$$
	\mbox{WF}(f) = \{(t,s)\subseteq \R^2 \times [-v,v]: \  \mathcal{SH}_\psi f (a,s,t)\mbox{ does not decay rapidly locally around }(t,s)\}
$$
for any tempered distribution $f$ with frequency support in $\C_{u,v}$ for $u,v>0$.
\end{theorem}

\begin{proof}
We have already proved this result for $f\in L^2(\C_{u,v})^\lor$. Since $\psi$ is a test function the generalization to tempered distributions follows easily by just repeating the same arguments.
\end{proof}

The following theorem has been proven in \cite{Labate2009} for very specific choices of $\psi$. As already stated, all we require is infinitely many vanishing moments in the $x_1$-direction.

\begin{theorem}[Resolution of the Wavefront Set II]\label{thm:rwf2}
Let $\psi$ be a Schwartz function with infinitely many vanishing moments in $x_1$-direction.
Let $f$ be a tempered distribution and $\mathcal{D}=\mathcal{D}_1\cup\mathcal{D}_2$, where
$\mathcal{D}_1 = \{(t_0,s_0)\in \R^2 \times [-1,1]:$ for $(s,t)$ in a neighbourhood $U$ of $(s_0,t_0)$, 
$|\mathcal{SH}_\psi f(a,s,t)| = O(a^k)$ for all $k\in \mathbb{N}$, with the implied constant uniform over $U\}$ and 
$\mathcal{D}_2 = \{(t_0,s_0)\in \R^2 \times (1,\infty]:$ for $(1/s,t)$ in a neighbourhood $U$ of $(s_0,t_0)$, 
$|\mathcal{SH}_{\psi^\nu} f(a,s,t)| = O(a^k)$ for all $k\in \mathbb{N}$, with the implied constant uniform over $U\}$.
Then 
$$
	\mbox{WF}(f)^c = \mathcal{D}.
$$
\end{theorem}

\begin{proof}
This is an immediate consequence of Theorems \ref{thm:direct} and \ref{thm:inverse} making the usual adaptions to handle general tempered distributions.
\end{proof}

%---------------------------------------------------------------
\section{Concluding remarks}
%---------------------------------------------------------------
In the present paper we have shown that there is a lot of freedom in choosing a shearlet. In particular this allows us to define a shearlet decomposition with respect to a compactly supported shearlet. Be it a tensor product wavelet or any directional derivative of a smooth function, our results show that any such function possesses the ability to resolve directional features of a given function.  For future work in this direction we would like to study the following:

\begin{itemize}
	\item {\bf Tight Frames.} As already suggested in Remark \ref{rem:tight}, it is possible to enformce a tight frame property by choosing an 			
	appropriate window function. We showed that for functions with frequency support in a cone this leads to useful window functions. 
	We want to pursue this approach in more detail with the goal of constructing compactly supported tight frames for $L^2(\R^2)$.
	\item {\bf Discretization.} We would like to discretize the frame construction using results in \cite{Fournasier2005} and study 								computational issues like inversion of the frame operator and so on.
	\item {\bf Sparsity.} For curvelet and shearlet transforms there exist several results confirming sparsity of certain objects (images, FIOs) 		in the respective representation \cite{guo2007,guo2008,candes2002, candes2002a,Candes2004}.
		We would like to extend these results to our general setting.
	\item {\bf Microlocal spaces.} In \cite{candes2003} the authors proved a characterization of microlocal Sobolev spaces via a curvelet square 		function. We would like to extend these results to our setting and also more general microlocal Besov spaces.
	\item {\bf Compactly supported curvelets?} Do functions with directional moments also serve as curvelets, in other words, is it possible to 		construct frames from functions with anisotropic moments by replacing the shear transform by rotations? This has actually been done in 				\cite{Smith1998}, see also the section on 'Hart Smith's transformation' in \cite{candes2003}. \end{itemize}

%-------------------------------------------------------------
\section{Acknowledgments}
%-------------------------------------------------------------
The research for this paper has been carried out while the author was working at the Center for Geometric Modeling and Scientific Visualization at KAUST, Saudi Arabia.
We thank Hans-Georg Feichtinger for several useful comments.
%\bibliographystyle{abbrv}
%\bibliography{Shearlets}

{\appendix
\section{Appendix}
\begin{proof}[Proof of Theorem \ref{thm:frame}]
To prove the theorem we need to estimate
\begin{equation}\label{eq:frame1}
	\int_{0<a < \Gamma ,\ |s|< \Xi} |\hat \psi \big(a\xi_1, \sqrt{a}(\xi_2 - s\xi_1)\big)|^2 a^{-3/2}da ds
\end{equation}
for $(\xi_1,\xi_2)\in \C_{u,v}$.
Notice that by assumption 
$$
	\int_{a \in \R_+ ,\ s \in \R} |\hat \psi \big(a\xi_1, \sqrt{a}(\xi_2 - s\xi_1)\big)|^2 a^{-3/2}da ds = C_\psi\quad \forall (\xi_1,\xi_2)\in 	\R^2
$$
and
$$
	\int_{a \in \R_+ ,\ s \in \R} |\hat \theta \big(a\xi_1, \sqrt{a}(\xi_2 - s\xi_1)\big)|^2 a^{-3/2}da ds = C_\theta\quad \forall 							(\xi_1,\xi_2)\in \R^2,
$$
where $\hat\psi(\xi_1,\xi_2) = \xi_1^\varepsilon \hat\theta (\xi_1,\xi_2)$.
Since for all $\xi \in \C_{u,v}$ we have $|\xi_2|\le v|\xi_1|$ and therefore the estimate 
\begin{equation}\label{eq:framecon}
	|\xi_2 - s\xi_1|\geq (|s|-v)|\xi_1|
\end{equation}
 holds. Furthermore, since
$\psi$ has Fourier decay of order $\tau$ in the second coordinate, we have the decay estimate
\begin{equation}\label{eq:framedec}
	|\hat \psi(\xi_1,\xi_2)| \le C(1+|\xi_2|^{\tau})^{-1}.
\end{equation} 
Now we estimate for any $\delta >0$
$$
	\int_{a\in \R,\ s> \Xi}|\hat \psi \big(a\xi_1, \sqrt{a}(\xi_2 - s\xi_1)\big)|^2 a^{-3/2}da ds = \sqrt{|\xi_1|}\int|\hat \psi \big(\tilde{a}
, \sqrt{\tilde{a}/|\xi_1|}(\xi_2 - s\xi_1)\big)|^2 \tilde{a}^{-3/2}d\tilde a ds
$$ 

$$ 
	= \underbrace{\int_{\tilde{a}< \delta}}_{A}  + \underbrace{\int_{\tilde{a} > \delta}}_{B}.
$$
The first term can be estimated by 
\begin{eqnarray}
	A & \le &  \delta^{2\varepsilon} \sqrt{|\xi_1|} \int|\hat \theta \big(\tilde{a},\sqrt{\tilde{a}/|\xi_1|}(\xi_2 - s\xi_1)\big)|^2 	\tilde{a}^{-3/2}d\tilde a ds \nonumber\\\nonumber
	 & \le & \delta^{2\varepsilon} \int_{\tilde a\in \R_+ , s\in \R}\sqrt{|\xi_1|} |\hat \theta \big(\tilde{a},\sqrt{\tilde{a}/|\xi_1|}(\xi_2 - 	s\xi_1)\big)|^2 \tilde{a}^{-3/2}d\tilde a ds \le  C_\theta \delta^{2\varepsilon}.
\end{eqnarray}
To estimate $B$ we use (\ref{eq:framedec}) and (\ref{eq:framecon}):
\begin{eqnarray}\nonumber
	B &=&\sqrt{|\xi_1|}\int_{\tilde a >\delta , |s|>\Xi}|\hat \psi \big(\tilde{a}, \sqrt{\tilde{a}/|\xi_1|}(\xi_2 - s\xi_1)\big)|^2 			\tilde{a}^{-3/2}d\tilde a ds \\
	\nonumber
	& \le & C\sqrt{|\xi_1|}\int_{\tilde a >\delta , |s|>\Xi}\big(\sqrt{\tilde a /|\xi_1|}|\xi_2 - s\xi_1|\big)^{-2\tau}\tilde{a}^{-3/2}d\tilde a 	ds\\ 
	\nonumber
	& \le &C\sqrt{|\xi_1|}\int_{\tilde a >\delta , |s|>\Xi}\big(\sqrt{\tilde a / |\xi_1|}(|s| - v)|\xi_1|\big)^{-2\tau}\tilde{a}^{-3/2}d\tilde a 	ds\\\nonumber
	&\le &C|\xi_1|^{ 1/2-\tau} \int_{\tilde{a} > \delta}\tilde{a}^{-3/2-\tau}d\tilde a \int_{ |s| > \Xi}(|s|-v)^{-2\tau}ds \\\nonumber
	&\le&
	Cu^{1/2-\tau} \int_{\tilde{a} > \delta}\tilde{a}^{-3/2-\tau}d\tilde a \int_{ |s| > \Xi}(|s|-v)^{-2\tau}ds \to 0
\end{eqnarray}
for $\Xi\to \infty$. We have used that $|\xi_1| > u$.
It is now clear that by choosing $\delta$ small and $\Xi$ large enough, the integral
$\int_{a\in \R,\ |s|> \Xi}|\hat \psi \big(a\xi_1, \sqrt{a}(\xi_2 - s\xi_1)\big)|^2 a^{-3/2}da ds $ can be made arbitrarily small \emph{uniformly} in $\xi \in \C_{u,v}$.
Next we use the Fourier decay in the first variable to estimate

\begin{eqnarray*}
	\int_{a> \Gamma,\ |s|\le \Xi}|\hat \psi \big(a\xi_1, \sqrt{a}(\xi_2 - s\xi_1)\big)|^2 a^{-3/2}da ds &\le &
	\int_{a> \Gamma,\ |s|\le \Xi}(a|\xi_1|)^{-\mu}a^{-3/2}da ds\\
	& \le & u^{-\mu}2\Xi \int_{a>\Gamma}a^{-3/2 - \mu}da \to 0
\end{eqnarray*} 
for $\Gamma \to \infty$. It follows that the integral $\int_{a> \Gamma,\ |s|\le \Xi}|\hat \psi \big(a\xi_1, \sqrt{a}(\xi_2 - s\xi_1)\big)|^2 a^{-3/2}da ds$ can also be made arbitrarily small \emph{uniformly} in $\xi \in \C_{u,v}$ by choosing $\Gamma$ large enough. 

Now we are ready to finish our proof:
\begin{eqnarray*}
	(\ref{eq:frame1}) &=& C_\psi - \int_{a\in \R,\ |s|> \Xi}|\hat \psi \big(a\xi_1, \sqrt{a}(\xi_2 - s\xi_1)\big)|^2 a^{-3/2}da ds \\
	& &
	- \int_{a> \Gamma,\ |s|\le \Xi}|\hat \psi \big(a\xi_1, \sqrt{a}(\xi_2 - s\xi_1)\big)|^2 a^{-3/2}da ds.
\end{eqnarray*}
By the previous discussion we can thus bound
$$
	C_\psi - \nu \le (\ref{eq:frame1}) \le C_\psi
$$
uniformly for $\xi \in \C_{u,v}$, any $\nu> 0$ and $\Gamma,\ \Xi$ large enough. This concludes the statement.

\end{proof}
\begin{lemma}[proof of (\ref{eq:diracdec})]Assume that $\psi$ has $M$ vanishing moments in $x_1$-direction and that $\hat \psi$,  $\omega_1^{-M}\hat \psi(\omega)$, $\big(\frac{\partial^L}{\partial x_2^L}\psi\big)^\land \in L^1(\R^2)$. With $\Phi$ a smooth bump function we have for any $1/2<\alpha < 1$ and $s\neq s_0$ that
$$
	\mathcal{SH}_\psi \Phi\delta_{x_1+s_0x_2-u}(a,s,t) = O(a^{(1-\alpha)M -3/4}+ a^{-3/4 + (\alpha - 1/2)L}\big).
$$
\end{lemma}

\begin{proof}
The proof is very similar to the proof of Theorem \ref{thm:direct}, the only difference is that we use the H\"older inequality instead of Cauchy-Schwarz. We assume without loss
of generality that $u=0$. It is well-known that the Fourier transform of $\delta_{x_1+s_0x_2}$ is given by $\delta_{x_1+1/s_0 x_2}$. This implies
that the tempered distrubution given by $\hat f:=(\Phi\delta_{x_1+s_0x_2-u})^\land = \hat \Phi \ast \delta_{x_1+1/s_0x_2}$ is actually a bounded function. 
Now we seperate the integral
$$
	\mathcal{SH}_\psi \Phi\delta_{x_1+s_0x_2-u}(a,s,t) = \langle \hat \psi_{ast}, \hat f\rangle
$$
into $A,B_1$ and $B_2$ in the same way as in the proof of Theorem \ref{thm:direct}. The estimate for $B_1$ given there is also valid for the assumptions of the present lemma
and we get $B_1 \le a^{\alpha N-3/4}\|\hat \psi\|_1$ for any $N\in \N$. Now we turn to an estimate for $A$. 
We write $\hat \psi (\xi_1 , \xi_2) = \xi_1^M\hat \theta (\xi_1,\xi_2)$.
We use $\hat \theta\in L^1(\R^2)$, and $\hat f\in L^\infty(\R^2)$ to estimate $A$ as
\begin{eqnarray}
	A&=&a^{3/4}\int_{|\xi_1|< a^{-\alpha}} |\hat f (\xi_1,\xi_2)| |\hat \psi \big( a\xi_1 , a^{1/2}(\xi_2 - s\xi_1)\big)|d\xi\nonumber \\
	&=& a^{3/4}\int_{|\xi_1|< a^{-\alpha}} a^M|\xi_1|^M|\hat f (\xi_1,\xi_2)| |\hat \theta \big( a\xi_1 , a^{1/2}(\xi_2 - s\xi_1)\big)|d\xi\nonumber\\
	&\le& a^{M(1-\alpha)}a^{3/4}\int_{|\xi_1|< a^{-\alpha}} |\hat f (\xi_1,\xi_2)| |\hat \theta \big( a\xi_1 , a^{1/2}(\xi_2 - s\xi_1)\big)|d\xi \nonumber\\\nonumber
	&\le& a^{M(1-\alpha)}a^{-3/4}\int_{\R^2}|\hat f(\xi_1,\xi_2)||\hat \theta(\tilde \xi_1,\tilde \xi_2)| d\tilde \xi_1d\tilde \xi_2\\\nonumber
	&\le&  a^{(1-\alpha)M-3/4} \|\hat \theta\|_1\|\hat f\|_\infty.
\end{eqnarray}
The estimate for $B_2$ is similar and we omit it.

\end{proof}
}
\end{document}